\documentclass[reqno,oneside,a4paper]{amsart}
\usepackage[hmargin={2.54cm,2.54cm},vmargin={3.17cm,3.17cm}]{geometry}
\usepackage{amsmath,amssymb,latexsym}
\usepackage{amsthm}
\usepackage{xcolor}
\usepackage{latexsym}
\usepackage[colorlinks=true,linkcolor=blue,citecolor=red,anchorcolor=green]{hyperref}
\allowdisplaybreaks
\theoremstyle{definition}
\newtheorem{Theorem}{Theorem}[section]
\newtheorem{lemma}[Theorem]{Lemma}

\newtheorem{Corollary}[Theorem]{Corollary}
\newtheorem{remark}[Theorem]{Remark}

\numberwithin{equation}{section}

\newcommand{\la}{\langle}
\newcommand{\ra}{\rangle}

\newcommand{\beq}{\begin{equation}}
\newcommand{\eeq}{\end{equation}}
\newcommand{\bes}{\begin{equation*}}
\newcommand{\ees}{\end{equation*}}
\newcommand{\rmi}{{\rm i}}
\newcommand{\Id}{{\rm Id}}

\newcommand{\lrp}[1]{\ensuremath{\left(#1\right)}}

\def\N{{\Bbb N}}

\def\R{{\Bbb R}}
\def\T{{\Bbb T}}

\hfuzz=\maxdimen
\def\N{{\Bbb N}}

\def\R{{\Bbb R}}
\def\T{{\Bbb T}}

\hfuzz=\maxdimen
\begin{document}
\title[Reducibility of 1-D quantum harmonic oscillator]{Reducibility of 1-D quantum harmonic oscillator with a new unbounded oscillatory perturbations}

\author{Xu Jin, Luo Jiawen, Wang Zhiqiang and Liang Zhenguo}
\address{School of Mathematical Sciences and
Key Lab of Mathematics for Nonlinear Science, Fudan University,
Shanghai 200433, China} 
\email{19110180005@fudan.edu.cn,  20110180010@fudan.edu.cn, 19110180010@fudan.edu.cn, zgliang@fudan.edu.cn}

\date{}

\begin{abstract} 
Enlightened by Lemma  1.7 in \cite{LiangLuo2021}, we prove a similar lemma which is 
based 
upon oscillatory integrals and Langer's  turning point theory.
 From it we show that the Schrödinger equation
\[
{\rm i}\partial_t u =-\partial_x^2 u+x^2 u+
\epsilon \langle x\rangle^\mu\sum_{k\in\Lambda}\left(a_k(\omega t)\sin(k|x|^\beta)+b_k(\omega t )\cos(k|x|^\beta)\right) u,
\quad u=u(t,x),~x\in\mathbb{R},~ \beta>1,
\]
can be reduced in $\mathcal{H}^1(\mathbb{R})$  to an autonomous system for most values of the frequency vector $\omega$,
where $\Lambda\subset\mathbb R\setminus\{0\}$,  $|\Lambda|<\infty$ and  $\langle x\rangle:=\sqrt{1+x^2}$.  The functions $a_k(\theta)$ and $b_k(\theta)$ are analytic on $\mathbb T^n_\sigma$ and $\mu\geq 0$ will be chosen according to the value of $\beta$. Comparing with \cite{LiangLuo2021},  the novelty is that the phase functions  of oscillatory integral are more degenerate when $\beta>1$. 
\end{abstract}
\maketitle 

\section{Introduction of the main results}
\subsection{Main theorem}
Following \cite{LiangLuo2021} we continue to consider the reducibility  for the time dependent Schr\"odinger  equation 
\begin{equation}\label{QHO}
\begin{gathered}
	\rmi\partial_t u=H_\epsilon(\omega t)u,\quad x\in\mathbb R,\\
	H_\epsilon:=-\partial_{xx}+x^2+\epsilon X(x,\omega t),
\end{gathered}
\end{equation}
where 
\begin{eqnarray*}
X(x,\theta)=\langle x\rangle^\mu\sum_{k\in\Lambda}\left(a_k(\theta)\sin(k|x|^\beta)+b_k(\theta)\cos(k|x|^\beta)\right)
\end{eqnarray*}
with $\Lambda\subset\mathbb R\setminus\{0\}$, $|\Lambda|<\infty$ and  $\la x\ra:=\sqrt{1+x^2}$. 
The functions 
$a_k(\theta)$ and $b_k(\theta)$ are analytic on $\mathbb T^n_\sigma=\{a+b\rmi\in\mathbb C^n/2\pi\mathbb Z^n:|b|<\sigma\}$ with $\sigma>0$ and $\beta>1$ and 
$\mu\geq 0$ will be chosen in the following.  We first introduce some functions and spaces.  \\
\indent
\textbf{Hermite functions}. The harmonic oscillator $T=-\partial_{xx}+x^2$ has eigenfunctions $(h_m)_{m\ge1}$, so called the Hermite functions, namely
\begin{equation}\label{HermiteFunction}
Th_m=(2m-1)h_m,\qquad \|h_m\|_{L^2(\mathbb R)}=1,\qquad m\ge1.
\end{equation}
\indent
\textbf{Linear spaces}. 
For $s\ge0$ denote by $\mathcal H^s$ the domain of $T^{\frac s2}$ endowed by the graph norm. For $s<0$, the space $\mathcal H^s$ is the dual of $\mathcal H^{-s}$. Particularly, for $s\ge0$ a integer we have
\[
\mathcal H^s=\{f\in L^2(\mathbb R):x^\alpha\partial^\beta f\in L^2(\mathbb R),~\forall~\alpha,\beta\in\mathbb N_0,~\alpha+\beta\le s\}.
\]
We also define the complex weighted-$\ell^2$-space
$
\ell_s^2:=\{\xi=(\xi_m\in\mathbb C,m\ge1):\sum_{m\ge1}m^s|\xi_m|^2<\infty\}.
$
To a function $u\in\mathcal H^s$ we associate the sequence $\xi$ of its Hermite coefficients by the formula $u=\sum_{m\ge1}\xi_mh_m(x)$. In the following we will identify the space $\mathcal H^s$ with $\ell_s^2$ by endowing both space the norm 
\[
\|u\|_{\mathcal H^s}=\|\xi\|_s=\left(\sum_{m\ge1}m^s|\xi_m|^2\right)^\frac12.
\]
Define 
\begin{eqnarray}\label{lbetamu}
l_*= l(\beta,\mu)=
\begin{cases}
	\frac{1}{4}\left(\frac\beta6-\mu\right),& 1<\beta<2,\\
	\frac{1}{4}\left(\frac{2}{9}-\mu\right),&\beta=2,\\
	\frac{1}{4}\left(\frac{\beta-2}{4\beta-2}-\mu\right),&\beta> 2.
\end{cases}
\end{eqnarray}

Then we can state our main theorem.
\begin{Theorem}\label{MainTheorem}
Assume $a_k(\theta)$ and $b_k(\theta)$ are analytic on $\mathbb T^n_\sigma$ with $\sigma>0$ and $\beta>1$ and $\mu$ satisfies 
\begin{equation}\label{muRange}
0\le\mu<
\begin{cases}
\frac\beta6,&1<\beta<2,\\
\frac29,&\beta=2,\\
\frac{\beta-2}{4\beta-2},&\beta>2.
\end{cases}
\end{equation}
\noindent There exists $\epsilon_*>0$ such that for all $0\le\epsilon<\epsilon_*$ there is a closed set $ D_\epsilon\subset D_0=[0,2\pi]^n$ of asymptotically full measure such that for all $\omega\in D_\epsilon$, the linear Schr\"odinger equation (\ref{QHO}) reduces to a linear autonomous equation in $\mathcal H^1$.\\ \indent
More precisely, for any $\omega\in D_\epsilon$ there exists a  linear isomorphism $\Psi_{\omega,\epsilon}^{\infty}(\theta)\in\mathfrak{L}(\mathcal H^{s'})$ with $0\le s'\le1$, analytically dependent on $\theta\in\mathbb T^n_{\sigma/2}$ and unitary on $L^2(\mathbb R)$, 
where  $\Psi_{\omega,\epsilon}^{\infty}-\Id\in\mathfrak{L}(\mathcal H^0,\mathcal H^{2l_{*}})\cap\mathfrak{L}(\mathcal H^{s'})$  and a bounded Hermitian operator $Q\in\mathfrak{L}(\mathcal H^1)$ such that $t\mapsto u(t,\cdot)\in\mathcal H^1$ satisfies \eqref{QHO} if and only if $t\mapsto v(t,\cdot)=\Psi_{\omega,\epsilon}^{\infty}u(t,\cdot)\in\mathcal H^1$ satisfies the autonomous equation 
\[
\rmi\partial_t v=-v_{xx}+x^2v+\epsilon Q(v),
\]
furthermore,  there are constants $C,K>0$ such that
\[
\begin{gathered}
	\text{Meas}( D_0\setminus D_\epsilon)\le C\epsilon^\frac{3l_*}{2(2l_*+5)(2l_*+1)},\\
	\|Q\|_{\mathfrak{L}(\mathcal H^p,\mathcal H^{p+4l_*})}+\|\partial_\omega Q\|_{\mathfrak{L}(\mathcal H^p,\mathcal H^{p+4l_*})}\le K,\quad\omega\in D_\epsilon,~p\in\mathbb N,\\
	\|\Psi_{\omega,\epsilon}^{\infty}(\theta)-\Id\|_{\mathfrak{L}(\mathcal H^0,\mathcal H^{2l_{*}})},\|\Psi_{\omega,\epsilon}^{\infty}(\theta)-\Id\|_{\mathfrak{L}(\mathcal H^{s'})}\le C\epsilon^\frac13,\quad(\omega,\theta)\in D_\epsilon\times\mathbb T^n_{\sigma/2}.
\end{gathered}
\]
\end{Theorem}
 Consequently,  Theorem \ref{MainTheorem} follows in the considered range of parameters the $\mathcal H^1$ norms of the solutions are all bounded forever and the spectrum of the corresponding operator is pure point.\\ \indent
\subsection{Related results and a critical lemma}
\indent  In the following we recall some relevant reducibility results. For 1-D quantum harmonic oscillators(`QHO' for short) with periodic or quasi-periodic in time bounded perturbations  see  \cite{Com87}, \cite{EV} , \cite{GT2011}, \cite{LiangWangZQ2021},  \cite{WL2017} and   \cite{Wang2008} . \\
\indent In \cite{BG2001} Bambusi and Graffi proved the reducibility of 1-D Schr\"odinger equation with an unbounded time quasiperiodic perturbation in which  the potential grows  at infinity like $|x|^{2l}$  with  a real $l>1$ and the perturbation is bounded by $1+|x|^{\beta}$ with $\beta<l-1$. The  reducibility in the limiting case $\beta= l-1$ was proved by Liu and Yuan in \cite{LY2010}. Recently, the results in \cite{BG2001} and \cite{LY2010} have been improved by Bambusi in \cite{Bam2017, Bam2018}, in which he firstly obtained the reducibility results for 1-D QHO  with unbounded perturbations.  \\
\indent It seems that the reducibility method in \cite{Bam2017, Bam2018}  is hard to be applied for 1-D Schr\"odinger equations with the  unbounded oscillatory perturbations(see remark 2.7 in \cite{Bam2017}).
The authors  \cite{LiangLuo2021}, \cite{LW2020} solved this problem  by Langer's turning point and oscillatory integral estimates. We remark that the critical step in \cite{LiangLuo2021} is
to build  up a decay estimate of  the integral  $\int_{\R}\la x\ra^{\mu} e^{{\rm i} kx} h_m(x) \overline{{h}_n(x)}dx$, in which the phase functions of oscillatory integral are
$\phi_{mn}(x): = \zeta_m(x) -\zeta_n(x)+kx,$
where 
$\zeta_m(x) =\int_{X_m}^x\sqrt{\lambda_m-t^2} dt$ with $X_m=\sqrt{2m-1}=\sqrt{\lambda_m}$.  \\
\indent Comparing with  \cite{LiangLuo2021},  in this paper 
the phase functions $\Psi_{mn}(x): = \zeta_m(x) -\zeta_n(x)+kx^{\beta}$ with $\beta>1$ are more degenerate.  For $1<\beta\leq 2$, we use a similar method as \cite{LiangLuo2021}. 
The most difficult part as \cite{LiangLuo2021} is the integral  
$\int_{X_m^{\frac23}}^{X_m-X_m^{\nu_2}}\la x\ra^\mu e^{{\rm i}kx^{\beta}}h_m(x)\overline{h_n(x)}dx$  where 
$\nu_2=1-\frac{\beta}{3}$ for $1<\beta<2$ and $\nu_2=\frac59$ when $\beta=2$.  We have to discuss  different cases in order  to obtain a suitable lower bound of the derivatives of the phase function. 
 For $\beta>2$ we find a new simple proof which follows from 
Corollary 3.2  in \cite{KT2005}, Lemma \ref{oscillatory integral lemma} and a straightforward computation. 
As Lemma 1.7 in \cite{LiangLuo2021}  we have the following. 
\begin{lemma}\label{CriticalLemma}
Assume $h_m(x)$ satisfies \eqref{HermiteFunction}. For any $k\neq 0$,
\[
\left|\int_{\R}\la x\ra^{\mu} e^{{\rm i} k|x|^{\beta}}h_m(x) \overline{{h}_n(x)}dx\right| \leq C \cdot C_{k,\beta}(mn)^{-l(\beta,\mu)},\quad m,n\ge1
\]
for some absolute constant $C>0$, where $\mu\ge0, \beta>1$,  $l(\beta,\mu)$ defined in (\ref{lbetamu}) and 
	\begin{gather*}
	C_{k,\beta}=
\begin{cases}
         |\beta(\beta-1)(\beta-2)k|^{-\frac13}\vee |k|^{-1}\vee |k|^\frac1{4-2\beta},& 1<\beta<2,\\
	|k|^{-1}\vee1,&\beta=2,\\
	|\beta k|^{-1}\vee 1,&\beta> 2.
\end{cases}\label{const}
	\end{gather*} 
\end{lemma}
\begin{remark}
In fact  $l(1,\mu)= \frac{1}{12}-\frac{\mu}{4}$ has been proved in \cite{LiangLuo2021}. 
\end{remark}

\indent In the end we review some relative results. Eliasson-Kuksin \cite{EK2009} initiated to prove the reducibility for PDEs in high dimension. See  \cite{GP2019} and \cite{LiangW2019} for higher-dimensional QHO with bounded potential. The  first reducibility result for n-D QHO was proved   in \cite{BGMR2018} by Bambusi-Gr\'ebert-Maspero-Robert.  Towards other PDEs with unbounded perturbations see the reducibility results by Montalto \cite{Mon19}  for linear wave equations  on $\mathbb T^d$ and Bambusi, Langella and Montalto \cite{BLM18}  for transportation equations(\cite{FGiMP19}).  Feola and Gr\'ebert \cite{FGr19} set up a reducibility result for  a linear Schr\"odinger equation on the sphere  $S^n$  with unbounded potential(\cite{FGN19}).\\
 \indent The reducibility results usually imply the boundedness of Sobolev norms. Delort \cite{Del2014} constructed a $t^{s/2}$- polynomial growth for 1-D QHO with certain time periodic  perturbation(\cite{Mas2018}). Basing on a Mourre estimate,  Maspero  \cite{Mas2021} proved similar results for 1-D QHO and half - wave equation on $\T$ and the instability is stable in some sense. For a polynomial periodic or quasi-periodic perturbations relative with 1-D QHO we refer to \cite{BGMR2018},  \cite{GrYa2000},  \cite{LZZ2020}  and \cite{LLZ2022}. For 2-D QHO with perturbation which is decaying in $t$, Faou-Rapha\"el \cite{FaRa2020} constructed a solution whose $\mathcal{H}^1-$norm presents logarithmic growth with t. For 2-D Schr\"odinger operator Thomann \cite{Th2020}  constructed explicitly  a traveling wave whose Sobolev norm presents polynomial growth with $t$, based on the study in \cite{ScTh2020} for linear Lowest Landau equations(LLL) with a time-dependent potential.  There are also many literatures, e.g. \cite{BLM2021, BGMR2019, BM2019, Bou99a, Bou99b, FZ12, MR2017, WWM2008}, which are closely relative to the upper growth bound of the solution in Sobolev space. \\
 \indent Our article is organized as follows: in Section 2 we state the  reducibility theorem, i.e. Theorem \ref{ReduTheorem}.  In Section 3, through checking all the assumptions in Theorem \ref{ReduTheorem} we prove Theorem \ref{MainTheorem}.  In Section 4 we prove Lemma \ref{CriticalLemma} for $1<\beta\leq 2$ and the case  for $\beta>2$ is delayed in Section 5.   Some auxiliary lemmas are presented in the Appendix.

 \textbf{Notation:}  We use the notations $\N_0=\{0,1,2,\cdots\}$, $\N=\{1,2,\cdots\}$, $\mathbb T^n=\mathbb R^n/2\pi \mathbb Z^n$ and  $\mathbb T^n_\sigma=\{a+b {\rm i}\in\mathbb C^n/2\pi\mathbb Z^n:|b|<\sigma\}$.
For Hilbert spaces $\mathcal H_1,\mathcal H_2$ we denote by $\mathfrak L(\mathcal H_1,\mathcal H_2)$ the space of bounded linear operators from $\mathcal H_1$ to $\mathcal H_2$ and write $\mathfrak{L}(\mathcal H_1,\mathcal H_1)$ as $\mathfrak{L}(\mathcal H_1)$ for simplicity.

\section{A KAM theorem}
Following \cite{EK2010, GT2011} we introduce the KAM Theorem from \cite{LiangW2019} especially for 1-D case.\\
\subsection{Setting} 
 \indent
Linear spaces. For $p\ge0$ we define $X_p:=\ell_p^2\times\ell_p^2=\{\zeta=(\zeta_a=(\xi_a,\eta_a)\in\mathbb C^2)_{a\in\mathbb N},  \|\zeta\|_p<\infty\}$ with $\|\zeta\|_p^2=\sum\limits_{a\in\mathbb N}a^p(|\xi_a|^2+|\eta_a|^2)$.  We equip the space with the symplectic structure $\rmi\sum\limits_{a\in\mathbb N}d\xi_a\wedge\eta_a$. \\ \indent 
Infinite matrices. Denote by $\mathcal M_\alpha$ the set of infinite matrices $A:\mathbb N\times\mathbb N\mapsto\mathbb C$ with the norm $|A|_\alpha:=\sup\limits_{a,b\in\mathbb N}(ab)^\alpha|A_a^b|<\infty$. We also denote $\mathcal M_\alpha^+$ be the subspace of $\mathcal M_\alpha$ satisfying that an infinite matrix $A\in\mathcal M_\alpha^+$ if $ |A|_{\alpha+}:=\sup\limits_{a,b\in\mathbb N}(ab)^\alpha(1+|a-b|)|A_a^b|<\infty$. \\ \indent
In fact one can prove that for all $\alpha>0$, a matrix in $\mathcal M_\alpha^+$ defines a bounded operator on $\ell_0^2$. However, when $\alpha\in(0,\frac12)$, we can't insure that  $\mathcal M_{\alpha}\subset \mathfrak{L}(\ell_0^2,\ell_s^2)$ for any $s\in \mathbb R$. 
This means that  $Px$ makes no sense when  the perturbation  operator $P\in\mathcal M_\alpha$ and $x\in\ell_0^2$. 
Fortunately, from Lemma 2.1 in \cite{GP2019} or Lemma 2.2 in \cite{LiangW2019} one can show $\mathcal M_{\alpha}\subset\mathfrak{L}(\ell_1^2,\ell_{-1}^2)$  and thus  the reducibility in $\mathcal{H}^1$ can be built up  in Theorem \ref{MainTheorem} instead of $L^2$.\\
\indent
Parameters. In this paper $\omega$ will play the role of a parameter belonging to $D_0=[0,2\pi]^n$. All the constructed maps will depend on $\omega$ with $\mathcal C^1$ regularity. When a map is only defined on a Cantor subset of $D_0$ the regularity is understood in Whitney sense. \\ \indent
A class of quadratic Hamiltonians. Let $D\subset D_0,\alpha>0$ and $\sigma>0$. We denote by ${\mathcal M}_\alpha(D,\sigma)$ the set of mappings as $\mathbb T^n_\sigma\times D\ni(\theta,\omega)\mapsto Q(\theta,\omega)\in\mathcal M_\alpha$ which is real analytic on $\theta\in\mathbb T^n_\sigma$ and $\mathcal C^1$ continuous on $\omega\in D$. And we endow this space with the norm
$[Q]_\alpha^{D,\sigma}:=\sup\limits_{\substack{\omega\in D,|\Im\theta|<\sigma\\ |k|=0,1}}|\partial_\omega^kQ(\theta,\omega)|_\alpha.$
The subspace of $\mathcal M_\alpha(D,\sigma)$ formed by $F(\theta,\omega)$ such that $\partial_\omega^kF(\theta,\omega)\in\mathcal M_\alpha^+,|k|=0,1$, is denoted by $\mathcal M_\alpha^+(D,\sigma)$ and endowed with the norm $[F]_{\alpha+}^{D,\sigma}:=\sup\limits_{\substack{\omega\in D,|\Im\theta|<\sigma\\ |k|=0,1}}|\partial_\omega^kF(\theta,\omega)|_{\alpha+}$. Besides, the subspace of $M_\alpha(D,\sigma)$ that are independent of $\theta$ will be denoted by $\mathcal M_\alpha(D)$ and for $N\in\mathcal M_\alpha(D)$, 
$$[N]_\alpha^D:=\sup\limits_{\omega\in D,|k|=0,1}|\partial_\omega^kN(\omega)|_\alpha. $$

\subsection{The reducibility theorem}
In this section we present an abstract reducibility theorem for a quadratic Hamiltonian quasiperiodic in time of the form
\begin{equation}\label{Hamiltonian}
H(t,\xi,\eta)=\langle\xi,N\eta\rangle+\epsilon\langle\xi,P(\omega t)\eta\rangle,\quad(\xi,\eta)\in X_1\subset X_0,
\end{equation}
and the corresponding Hamiltonian system
\begin{equation}
\label{fangcheng2}
\begin{cases}
\dot\xi=-\rmi N\xi-\rmi\epsilon P^T(\omega t)\xi,\\
\dot\eta=\rmi N\eta+\rmi\epsilon P(\omega t)\eta,
\end{cases}
\end{equation}
where $N=\text{diag}\{\lambda_a,a\in\mathbb N\}$ satisfies the following spectrum assumptions:\\ \indent
\textbf{Hypothesis A1-Asymptotics}. There exist positive constants $c_0,c_1,c_2$ such that
\[
c_1a\ge\lambda_a\ge c_2a\text{ and }|\lambda_a-\lambda_b|\ge c_0|a-b|,~\forall~a,b\in\mathbb N.
\]
\indent
\textbf{Hypothesis A2-Second Melnikov condition in measure estimates}. There exist positive constants $\alpha_1,\alpha_2$ and $c_3$ such that the following holds: for each $0<\kappa<\frac14$ and $K>0$ there exists a closed subset $D'=D'(\kappa,K)\subset D_0$ with $\text{Meas}(D_0\setminus D')\le c_3K^{\alpha_1}\kappa^{\alpha_2}$ such that for all $\omega\in D',k\in\mathbb Z^n$ with $0<|k|\le K$ and $a,b\in\mathbb N$ we have $|k\cdot\omega+\lambda_a-\lambda_b|\ge\kappa(1+|a-b|)$.
\\ \indent Then we can state our reducibility results.
\begin{Theorem}\label{ReduTheorem}
Given a nonautonomous Hamiltonian \eqref{Hamiltonian}, we assume that $(\lambda_a)_{a\in\mathbb N}$ satisfies Hypothesis A1-A2 and $P(\theta)\in\mathcal M_{\alpha}(D_0,\sigma)$ with $\alpha,\sigma>0$. Let $\gamma_1=\max\{\alpha_1,n+3\}$ and $\gamma_2=\frac{\alpha\alpha_2}{2\alpha\alpha_2+5}$, then there exists $\epsilon_*>0$ such that for all $0\le\epsilon<\epsilon_*$ there are
\\ \indent
(i) a Cantor set $D_\epsilon\subset D_0$ with $\text{Meas}(D_0\setminus D_\epsilon)\le C\epsilon^{\frac{3\delta\alpha}{2\alpha+1}}$ for a $\delta\in(0,\frac{\gamma_2}{24})$;\\ \indent
(ii) a $\mathcal C^1$ family in $\omega\in D_\epsilon$(in Whitney sense), linear, unitary, analytically dependent on $\theta\in\mathbb T^n_{\sigma/2}$ and symplectic coordinate transformation $\Phi_\omega^\infty(\theta): X_0\mapsto X_0,(\omega,\theta)\in D_\epsilon\times\mathbb T^n_{\sigma/2}$, of the form 
\[
(\xi_+,\eta_+)\mapsto(\xi,\eta)=\Phi_\omega^\infty(\theta)(\xi_+,\eta_+)=(\overline M_\omega(\theta)\xi_+,M_\omega(\theta)\eta_+),
\]
where $\Phi_\omega^\infty(\theta)-\Id$ satisfies for $0\le s'\le1$
\[
\|\Phi_\omega^\infty(\theta)-\Id\|_{\mathfrak{L}(X_0,X_{2\alpha})},\|\Phi_\omega^\infty(\theta)-\Id\|_{\mathfrak{L} (X_{s'})}\le C\epsilon^{\frac13};
\]
\indent
(iii) a $\mathcal C^1$ family of autonomous quadratic Hamiltonian in normal forms
\[
H_\infty(\xi_+,\eta_+)=\langle\xi_+,N_\infty(\omega)\eta_+\rangle=\sum_{m\ge1}\lambda_m^\infty\xi_{+,m}\eta_{+,m},\quad\omega\in D_\epsilon,
\]
where $N_\infty(\omega)=\text{diag}\{\lambda^{\infty}_m,m\in \mathbb N \}$ is diagonal and is close to $N$ in the sense of
\[
[N_\infty(\omega)-N]_\alpha^{D_\epsilon}\le C\epsilon,
\]
such that 
\[
H(t,\Phi_\omega^\infty(\omega t)(\xi_+,\eta_+))=H_\infty(\xi_+,\eta_+),~t\in\mathbb R,~(\xi_+,\eta_+)\in X_1,~\omega\in D_\epsilon.
\]
\end{Theorem}
\section{Application to the quantum harmonic oscillator}
\indent In this section we will prove Theorem \ref{MainTheorem} by applying Theorem \ref{ReduTheorem} to the original equation \eqref{QHO}.  Following the strategies in \cite{EK2009}, we expand $u$ on the Hermite basis $(h_m)_{m\ge1}$ as well as $\bar u$ by the following formula
\[
u=\sum_{m\ge1}\xi_mh_m,\quad\bar{u}=\sum_{m\ge1}\eta_m\bar{h}_m.
\]
Therefore the equation \eqref{QHO} is equivalent to a nonautonomous Hamiltonian system 
\begin{equation}\label{Hamiltonianapp}
\begin{cases}
\dot\xi_m=-\rmi\frac{\partial H}{\partial\eta_m}=-\rmi(2m-1)\xi_m-
\rmi\epsilon\left(P^T(\omega t)\xi\right)_m,\\
\dot\eta_m=\rmi\frac{\partial H}{\partial\xi_m}=\rmi(2m-1)\eta_m+\rmi\epsilon\left(P(\omega t)\eta\right)_m,
\end{cases}\quad m\ge1,
\end{equation}
where 
\begin{equation*}
H(t,\xi,\eta)=\langle\xi,N\eta\rangle+\epsilon\langle\xi,P(\omega t)\eta\rangle,\quad(\xi,\eta)\in X_1\subset X_0,
\end{equation*}
and $N=\text{diag}\{2m-1,m\ge1\}$ and 
\begin{equation}\label{DefinitionP}
P_m^n(\omega t)=\sum\limits_{k\in \Lambda} a_k(\omega t) \int_{\mathbb R} \langle x\rangle^\mu\sin k|x|^\beta h_m(x)\overline{h_n(x)}dx +
\sum\limits_{k\in \Lambda} b_k(\omega t) \int_{\mathbb R} \langle x\rangle^\mu\cos  k |x|^\beta h_m(x)\overline{h_n(x)}dx,
\end{equation}
where the frequencies $\omega\in D_0=[0,2\pi]^n$ are  the external parameters. 
\\ \indent
The spectrum assumptions can be easily checked by the following two lemmas. 
\begin{lemma}\label{A1}
When $\lambda_a=2a-1,a\in\mathbb N$, Hypothesis A1 holds true with $c_0=c_2=1$ and $c_1=2$.
\end{lemma}
\begin{lemma}\label{A2}
	When $\lambda_a=2a-1,a\in\mathbb N$, Hypothesis A2 holds true with $\alpha_1=n+1,\alpha_2=1,c_3=c(n)$ and $ D_0=[0,2\pi]^n$,
\[
 D'=\{\omega\in[0,2\pi]^n:|k\cdot\omega+j|\ge\kappa(1+|j|),~\forall~j\in\mathbb Z,~k\in\mathbb Z^n\setminus\{0\}\}.
\]
\end{lemma}
\indent
The following lemma is a direct corollary of Lemma \ref{CriticalLemma}.
\begin{lemma}\label{Malpha}
Assume that $a_k(\theta)$ and $b_k(\theta)$ are  analytic on $\mathbb T^n_\sigma$ for any nonzero $k\in \Lambda$ with $\sigma>0$ and $\beta>1$ and $\mu$ satisfies \eqref{muRange}, then there exists $\alpha=l(\beta,\mu)>0$ such that the matrix function $P(\theta)$ defined by \eqref{DefinitionP} is analytic from $\mathbb T^n_\sigma$ into $\mathcal M_\alpha$.
\end{lemma}
\noindent Proof of Theorem \ref{MainTheorem}. Expanding the Hermite basis $(h_m)_{m\ge1}$, the Schr\"odinger equation \eqref{QHO} becomes Hamiltonian system \eqref{Hamiltonianapp}, which is the form of equation \eqref{fangcheng2} with $\lambda_a=2a-1$. By lemmas above, we can apply Theorem \ref{ReduTheorem} to \eqref{Hamiltonianapp} with $\gamma_1=n+3,\gamma_2=\frac\alpha{2\alpha+5}$ and $\delta=\frac{\gamma_2}{48}$. This follows Theorem \ref{MainTheorem}.
\\ \indent
More precisely, in new variables given in Theorem \ref{ReduTheorem}, $(\xi,\eta)=(\overline M_\omega\xi_+,M_\omega\eta_+)$, system \eqref{Hamiltonianapp} is conjugated into an autonomous system of the form:
\[
\begin{cases}
\dot\xi_{+,a}=-\rmi\lambda_a^\infty(\omega)\xi_{+,a},\\
\dot\eta_{+,a}=\rmi\lambda_a^\infty(\omega)\eta_{+,a},
\end{cases}\quad a\in\mathbb N.
\]
Therefore the solution subject to the initial datum $(\xi_+(0),\eta_+(0))$ reads
\[
(\xi_+(t),\eta_+(t))=(e^{-\rmi tN_\infty}\xi_+(0),e^{\rmi tN_\infty}\eta_+(0)),\quad t\in\mathbb R,
\]
where $N_\infty=\text{diag}\{\lambda^{\infty}_a, \ a\geq 1\}$. Then the solution of \eqref{QHO} with the initial datum $u_0(x)=\sum_{a\ge1}\xi_a(0)h_a(x)\in\mathcal H^1$ is formulated by $u(t,x)=\sum_{a\ge1}\xi_a(t)h_a(x)$ with $\xi(t)=\overline M_\omega(\omega t)e^{-\rmi tN_\infty}M_\omega^T(0)\xi(0)$, where we use the fact $\left(\overline M_\omega\right)^{-1}=M_\omega^T$ since $M$ is unitary.
\\ \indent
Now we define the coordinate transformation $\Psi_\omega^{\infty}(\theta)$ by
\[
\Psi_\omega^{\infty}(\theta)\left(\sum_{a\ge1}\xi_ah_a(x)\right):=\sum_{a\ge1}\left(M_\omega^T(\theta)\xi\right)_ah_a(x)=\sum_{a\ge1}\xi_{+,a}h_a(x).
\]
Then we have $u(t,x)$ satisfies \eqref{QHO} if and only if $v(t,x)=\Psi_\omega^{\infty}(\omega t)u(t,x)$ satisfies the autonomous equation $\rmi\partial_tv=-v_{xx}+x^2v+\epsilon Q(v)$, where 
\[
\epsilon Q\left(\sum_{a\ge1}\xi_ah_a(x)\right)=\sum_{a\ge1}((N_\infty-N_0)\xi)_ah_a(x)=\sum_{a\ge1}(\lambda_a^\infty-\lambda_a)\xi_ah_a(x).
\]
The rest estimates are standard (see Lemma 3.4 in \cite{LiangLuo2021} for the details). \qed

\section{Proof of Lemma \ref{CriticalLemma} when  $1<\beta\leq 2$}\label{S4}
For reader's convenience, we will use the notations in \cite{LiangLuo2021}.  In the whole section we will always suppose $\mu\geq 0$ and don't point it out in the following lemmas. \\
\indent The eigenfunction of the quantum oscillator operator $T$ is 
$h_n(x)={(n! 2^n \pi^{\frac12})^{-\frac12}}{e^{-\frac12 x^2} H_n(x)}$, where $H_n(x)$ is the $n-th$ Hermite polynomial. Since $h_n(x)$ is an even (or odd) function when $n$ is odd (or even), we only need to estimate 
\begin{equation}\label{int0infty}
	\int_0^{+\infty} \la x\ra^{\mu} e^{{\rm i}kx^{\beta}}h_m(x)\overline{h_n(x)}dx,\quad 1\leq m\leq n.
\end{equation}
By Lemma 4.4 and Remarks 4.5, 4.6 in \cite{LiangLuo2021}, when $m>m_0$,
\begin{eqnarray*}
h_m(x)&=&
(\lambda_m-x^2)^{-\frac{1}{4}}(\frac{\pi\zeta_m}{2})^{\frac{1}{2}}H_{\frac{1}{3}}^{(1)}(\zeta_m) +(\lambda_m-x^2)^{-\frac{1}{4}}(\frac{\pi\zeta_m}{2})^{\frac{1}{2}}H_{\frac{1}{3}}^{(1)}(\zeta_m)O (\frac{1}{\lambda_m}) \nonumber\\
&:=& \psi^{(m)}_1(x) +\psi^{(m)}_2(x),
\end{eqnarray*}
where $\zeta_m(x)=\displaystyle\int_{X_{m}}^x\sqrt{\lambda_m-t^2}dt$ with $X_m^2=\lambda_m(X_m>0)$.
Otherwise, when $m\le m_0$, then 
$h_m(x)=\psi_1^{(m)}(x)+\psi_2^{(m)}(x)$ for $x>2X_{m_0}$, where $\psi_1^{(m)}(x)=(\lambda_m-x^2)^{-\frac14}(\frac{\pi\zeta_m}{2})^\frac12 H_\frac13^{(1)}(\zeta_m)$ and $|\psi_2^{(m)}(x)|\le\frac{C}{x^2}|\psi_1^{(m)}(x)|$. Following the same strategies in \cite{LiangLuo2021} we distinguish 3 cases to estimate (\ref{int0infty}): \\
I. $m,n< C_0: = 2^8m_0^3$;\\
II. $m\le m_0$ and $n\ge C_0$;\\
III. $m,n>m_0$.
\subsection{the estimates for Case I and Case II}

\begin{lemma}
When $n, m<C_0$,
$$\left|\int_0^{+\infty} \la x\ra^{\mu} e^{{\rm i}kx^{\beta}}h_m(x)\overline{h_n(x)}dx\right|\le \frac{C}{(mn)^{\frac14(\frac{\beta}{6}-\mu)}}. $$
\end{lemma}

\begin{proof}
When $x\le X_0$, from $\rm{H\ddot{o}lder}$ inequality and $n,m<C_0$, we have 
$$\left|\int_0^{X_0} \la x\ra^{\mu} e^{{\rm i}kx^{\beta}}h_m(x)\overline{h_n(x)}dx\right|\le X_0^\mu\le  \frac{C}{(mn)^{\frac14(\frac{\beta}{6}-\mu)}}.$$
where $X_0$ is a positive constant depending on $C_0$ only. 
When $x>X_0$, $|X_m^2-x^2|^{-\frac14}<1$, we have $\left|\sqrt{\frac{\pi \zeta_m}{2}}H^{(1)}_\frac{1}{3}(\zeta_m)\right|\leq e^{-\left| \zeta_m\right|}$ by  Lemma 5.4 in \cite{LiangLuo2021}.  By Lemma 5.5 in \cite{LiangLuo2021} we have 
$\left|\zeta_m\right| \geq\frac{2\sqrt2}{3}X_m^\frac12(x-X_m)^\frac32\geq x-X_0$ for  $x> X_0$. 
Thus
\begin{align*}
   \left|\int_{X_0}^{+\infty} \la x\ra^{\mu} e^{{\rm i}kx^{\beta}}h_m(x)\overline{h_n(x)}dx\right|
\le \int_{X_0}^{+\infty} \la x\ra^{\mu}e^{-2(x-X_0)}dx\le Ce^{2X_0}\le
\frac{C}{(mn)^{\frac14(\frac{\beta}{6}-\mu)}}. 
\end{align*}
\end{proof}

\begin{lemma}
For $m\le m_0$ and $n\ge C_0$ and $\mu\geq 0$, 
$$\left|\int_0^{+\infty} \la x\ra^{\mu} e^{{\rm i}kx^{\beta}}h_m(x)\overline{h_n(x)}dx\right|\le \frac{C}{(mn)^{\frac14(\frac{\beta}{6}-\mu)}}.$$
\end{lemma}
\begin{proof}
We divide the integral into two parts. 
$$\int_0^{+\infty} \la x\ra^{\mu} e^{{\rm i}kx^{\beta}}h_m(x)\overline{h_n(x)}dx=\int_0^{X_n^\frac13}+\int_{X_n^\frac13}^{+\infty}.$$
Since $x>2X_{m_0}$, we have 
$$|h_m(x)|\le 2(x^2-X_m^2)^{-\frac14}|\sqrt{\frac{\pi \zeta_m}{2}}H^{(1)}_\frac{1}{3}(\zeta_m)|\le 2 e^{-|\zeta_m|}.$$
On the other hand, for $x\in [0,X_n^\frac13]$, one has $|h_n(x)|\le C(X_n^2-x^2)^{-\frac14}$.  Note $1<\beta<2$, it follows 
\begin{eqnarray*}
\left|\int_0^{X_n^\frac13}\la x\ra^{\mu} e^{{\rm i}kx^{\beta}}h_m(x)\overline{h_n(x)}dx\right|& \le&  C\int_0^{X_n^\frac13}\la x\ra^{\mu}(X_n^2-x^2)^{-\frac14}dx \le CX_n^{-\frac16+\frac\mu3}\\
& \le & \frac{C}{(mn)^{\frac14(\frac{\beta}{6}-\mu)}}.
\end{eqnarray*}
When $x\ge X_n^\frac13\ge2X_{m_0}$, from Lemma 5.5 in \cite{LiangLuo2021}, $e^{-|\zeta_m|}\le e^{-C(x-X_m)}$. Note $\|h_n(x)\|_{L^2}=1$, from $\rm{H\ddot{o}lder}$ inequality, 
$$\left|\int_{X_n^\frac13}^{+\infty}\la x\ra^{\mu} e^{{\rm i}kx^{\beta}}h_m(x)\overline{h_n(x)}dx\right|\le C\left(\int_{X_n^\frac13}^{+\infty}\la x\ra^{2\mu}e^{-Cx}dx\right)^\frac12\le e^{-CX_n^\frac13}.$$
\end{proof}
\subsection{The estimate for Case III}
In the following we will turn to the complicated case when $m,n>m_0$. We divide the integral into two parts $\int_0^{+\infty} \la x\ra^{\mu} e^{{\rm i}kx^{\beta}}h_m(x)\overline{h_n(x)}dx=\int_0^{X_n}+\int_{X_n}^{+\infty}$. 
We first go to the latter case  $\int_{X_n}^{+\infty}$.  \\
\subsubsection{The integral on $[X_n, +\infty)$}
\begin{lemma}\label{Xntowuqiong}
For $m_0<m\le n$, 
$$\displaystyle\left|\int_{X_n}^{+\infty}\la x\ra^\mu e^{{\rm i}kx^{\beta}}h_m(x)\overline{h_n(x)}dx\right|\leq\frac{C}{m^{\frac{1}{12}-\frac\mu4} n^{\frac{1}{12}-\frac\mu4}}\leq \frac{C}{(mn)^{\frac14(\frac{\beta}{6}-\mu)}}.$$
\end{lemma}
We first estimate the integral on $[2X_n, +\infty]$.  The following result is clear from \cite{LiangLuo2021}. 
\begin{lemma}\label{2.2}
For $m_0<m\le n$,
$$\displaystyle\left|\int_{2X_n}^{+\infty}\la x\ra^\mu e^{{\rm i}kx^\beta}h_m(x) \overline{h_n(x)}dx\right|\leq e^{-C n}.$$
\end{lemma}

For the integral on $[X_n,2X_n]$, we prove that
\begin{lemma}\label{2.3}
For $m_0<m\le n$,
$$\displaystyle\left|\int^{2X_n}_{X_n}\la x\ra^\mu e^{{\rm i}kx^{\beta}}h_m(x) \overline{h_n(x)}dx\right|\leq  \frac{C}{m^{\frac{1}{12}-\frac\mu4} n^{\frac{1}{12}-\frac\mu4}}.$$
\end{lemma}
\begin{proof} As \cite{LiangLuo2021},  we only need to estimate the following integral $I: = \int^{2X_n}_{X_n}\la x\ra^\mu e^{{\rm i}kx^{\beta}}\psi_1^{(m)}(x) \overline{\psi_1^{(n)}(x)}dx$
since the rest three ones are higher order.  
 $I$ can be divided into two parts as 
\begin{align*}
&|I|=\displaystyle\left|\Bigg(\int^{2X_n}_{X_n+X_n^{\frac{1}{3}}}+\int^{X_n+X_n^{\frac{1}{3}}}_{X_n}\Bigg)\la x\ra^\mu e^{{\rm i}kx^{\beta}}\psi_1^{(m)}(x) \overline{\psi_1^{(n)}(x)}dx\right|\\
\leq& CX_n^\mu \Big(\int^{2X_n}_{X_n+X_n^{\frac{1}{3}}}+\int^{X_n+X_n^{\frac{1}{3}}}_{X_n}\Big)\left|\psi_1^{(m)}(x) \overline{\psi_1^{(n)}(x)}\right|dx.
\end{align*}
From Lemma 5.5 in \cite{LiangLuo2021}, when $x\geq X_n+X_n^{\frac{1}{3}}$,
$\left|\zeta_n\right| \geq\frac{2\sqrt2}{3}X_n^\frac{1}{2}(x-X_n)^\frac{3}{2}\geq\frac{2\sqrt2}{3}X_n.$
Thus 
\begin{align*}
          &\int^{2X_n}_{X_n+X_n^{\frac{1}{3}}}\left|\psi_1^{(m)}(x) \overline{\psi_1^{(n)}(x)}\right| dx\\
\leq & C\int^{2X_n}_{X_n+X_n^{\frac{1}{3}}}(x^2-\lambda_m)^{-\frac{1}{4}}(x^2-\lambda_n)^{-\frac{1}{4}} e^{-\left| \zeta_n\right|}dx \\
\leq & Ce^{-\frac{2\sqrt2}{3}X_n} \int^{2X_n}_{X_n+X_n^{\frac{1}{3}}}(x^2-\lambda_n)^{-\frac{1}{2}}dx\le  Ce^{-\frac{2\sqrt2}{3}X_n}.
\end{align*}
For the second part, 
\begin{align*}
          &\int^{X_n+X_n^{\frac{1}{3}}}_{X_n}\left|\psi_1^{(m)}(x) \overline{\psi_1^{(n)}(x)}\right|dx\\
\leq & C\int^{X_n+X_n^{\frac{1}{3}}}_{X_n}(x^2-\lambda_m)^{-\frac{1}{4}}(x^2-\lambda_n)^{-\frac{1}{4}}dx \leq C\int^{X_n+X_n^{\frac{1}{3}}}_{X_n} (x^2-\lambda_n)^{-\frac{1}{2}}dx \\
\leq &CX_n^{-\frac{1}{2}} \int^{X_n+X_n^{\frac{1}{3}}}_{X_n}(x-X_n)^{-\frac{1}{2}}dx \leq CX_n^{-\frac{1}{3}}.
\end{align*}
It follows 
$|I|\leq CX_n^{\mu-\frac13} \leq\frac{C}{m^{\frac{1}{12}-\frac\mu4}n^{\frac{1}{12}-\frac\mu4}}$.
\end{proof}
Combining with the above two lemmas we finish Lemma \ref{Xntowuqiong}. 

In the following we will estimate the integral on $[0,X_n]$, which is the most complicated case. Note $m_0< m\leq n$, the following two cases have to be considered respectively:  I. $X_n> 2X_m$; II. $X_m\leq X_n \leq 2X_m$.

\subsubsection{the integral estimate on $[0, X_n]$ when $X_n> 2X_m$}
Our aim in this part is to build the following 
\begin{lemma}\label{lemma2.4}
For $k \neq 0$, if $X_n> 2X_m$ and $1<\beta\leq 2$, then 
$$\displaystyle\left|\int^{X_n}_0 \la x\ra^\mu e^{{\rm i}kx^{\beta}}h_m(x)\overline{h_n(x)}dx\right|\leq
\frac{C (|k|^\iota\vee 1)}{m^{\frac18-\frac\mu4} n^{\frac{1}{12}-\frac\mu4}},$$
where $m_0< m\leq n$ and $\iota=\begin{cases}
\frac1{4-2\beta},&1<\beta<2,\\
0,&\beta=2.
\end{cases}$
\end{lemma}

 As \cite{LiangLuo2021} we will use the following notation in the remained parts. We denote 
$f_m(x)=\int_0^{\infty}e^{-t}t^{-\frac{1}{6}}\lrp{1+\frac{{\rm i}t}{2\zeta_m}}^{-\frac{1}{6}}dt$
and 
$f_n(x)=\int_0^{\infty}e^{-t}t^{-\frac{1}{6}}\lrp{1+\frac{{\rm i}t}{2\zeta_n}}^{-\frac{1}{6}}dt$. 
When $x\in [0,X_{m}]$, from a straightforward computation we have 
\begin{align*}
\psi_1^{(m)}(x) &= (X_m^2-x^2)^{-\frac{1}{4}}\sqrt{\frac{\pi\zeta_m}{2}} H_\frac{1}{3}^{(1)}(\zeta_m)\\
&= (X_m^2-x^2)^{-\frac{1}{4}} \frac{e^{{\rm i}\lrp{\zeta_m-\frac{\pi}{6}-\frac{\pi}{4}}}}{\Gamma{\lrp{\frac{5}{6}}}}
\int_0^{\infty}e^{-t}t^{-\frac{1}{6}}\lrp{1+\frac{{\rm i}t}{2\zeta_m}}^{-\frac{1}{6}}dt\\
&=C(X_m^2-x^2)^{-\frac{1}{4}}e^{{\rm i}\zeta_m(x)}f_m(x).
\end{align*}
Similarly, 
$
\overline{\psi_1^{(n)}(x)} =C(X_n^2-x^2)^{-\frac{1}{4}}e^{-{\rm i}\zeta_n(x)}\overline{f_n(x)}.
$
For $x\in [0,X_m]$,  denote $\Psi(x)=(X_m^2-x^2)^{-\frac{1}{4}} (X_n^2-x^2)^{-\frac{1}{4}}\cdot f_m(x)\overline{f_n(x)}$
and 
$g(x)=(\zeta_n(x)-\zeta_m(x)-kx^{\beta})^\prime=\sqrt{X_n^2-x^2}-\sqrt{X_m^2-x^2}-k\beta x^{\beta-1}$,then 
\begin{align*}
\Psi^\prime(x)=&\frac{1}{2}x(X_m^2-x^2)^{-\frac{5}{4}} (X_n^2-x^2)^{-\frac{1}{4}}\cdot f_m(x)\overline{f_n(x)}\\
+&\frac{1}{2}x(X_m^2-x^2)^{-\frac{1}{4}} (X_n^2-x^2)^{-\frac{5}{4}}\cdot f_m(x)\overline{f_n(x)}\\
+&(X_m^2-x^2)^{-\frac{1}{4}} (X_n^2-x^2)^{-\frac{1}{4}}\cdot \left(f_m^\prime(x)\overline{f_n(x)}+f_m(x)\overline{f_n^\prime(x)}\right).
\end{align*}
When $x\in [0, X_m]$,  
 $\left|f_m(x)\right| \leq \Gamma(\frac{5}{6})$ and $\left|f_n(x)\right| \leq \Gamma(\frac{5}{6})$.  Thus, 
\begin{Corollary}\label{coro2.5}
For $x\in [0,X_m)$ and $m\leq n$, 
\begin{align*}
\left| \Psi^\prime(x)\right|\leq & C\Big(x(X_m^2-x^2)^{-\frac{5}{4}}(X_n^2-x^2)^{-\frac{1}{4}} +x(X_m^2-x^2)^{-\frac{1}{4}} (X_n^2-x^2)^{-\frac{5}{4}}+\\
& \frac{(X_m^2-x^2)^{\frac{1}{4}} (X_n^2-x^2)^{-\frac{1}{4}}}{X_m(X_m-x)^3} +
\frac{(X_m^2-x^2)^{-\frac{1}{4}} (X_n^2-x^2)^{\frac{1}{4}}}{X_n (X_n-x)^3}\Big)\\
=&C\big(J_1+J_2+J_3+J_4\big)\leq C(J_1+J_3).
\end{align*}
\end{Corollary}
We first estimate the integral on $[0, X_m-X_m^{-\frac{1}{3}}]$. 

\begin{lemma}\label{lemma2.6}
For $k\neq 0,1<\beta<2$, if $X_n> 2X_m$, then
$$\left|\int_0^{X_m-X_m^{-\frac{1}{3}}} \la x\ra^\mu e^{{\rm i}kx^{\beta}}h_m(x)\overline{h_n(x)}dx\right|\leq \frac{C(|k|^\frac1{4-2\beta}\vee 1)}{m^{\frac{1}{8}-\frac\mu4} n^{\frac{1}{8}-\frac\mu4}}, $$
where $m_0< m\leq n$.
\end{lemma}
\begin{proof}
First we estimate the main term of the integral
\begin{align*}
\int^{X_m-X_m^{-\frac{1}{3}}}_{0} \la x\ra^\mu e^{{\rm i}kx^{\beta}}\psi_1^{(m)}(x)\overline{\psi_1^{(n)}(x)}dx= C\int^{X_m-X_m^{-\frac{1}{3}}}_{0} \la x\ra^\mu e^{{\rm i}(\zeta_m-\zeta_n+kx^{\beta})}\Psi(x)dx,\\
\end{align*}
by method of oscillating integral estimate, where
$\Psi(x)=(X_m^2-x^2)^{-\frac{1}{4}} (X_n^2-x^2)^{-\frac{1}{4}}\cdot f_m(x)\overline{f_n(x)}.$
We discuss two different cases. \\
Case 1:  $k\leq\frac{X_n^{2-\beta}}{8}$. In this case, we have
\begin{align*}
g(x)\geq\sqrt{X_n^2-x^2}-\sqrt{\frac{X_n^2}{4}-x^2}-\beta kX_n^{\beta-1}
\geq \frac{X_n}{2}-2\cdot\frac{X_n^{2-\beta}}{8}X_n^{\beta-1}\geq\frac{X_n}{4}.
\end{align*}
Thus, by Lemma \ref{oscillatory integral lemma},
\begin{align*}
&\left|\int^{X_m-X_m^{-\frac{1}{3}}}_{0}e^{{\rm i}\frac{\zeta_m-\zeta_n+kx^\beta}{X_n}X_n}\la x\ra^\mu\Psi(x)dx\right|\\
\leq& CX_n^{-1}\left(\left|\left(\la x\ra^\mu\Psi\right)(X_m-X_m^{-\frac{1}{3}})\right|+\int_0^{X_m-X_m^{-\frac{1}{3}}}\left|\left(\la x\ra^\mu\Psi\right)^\prime(x)\right|dx\right)\\
\leq& CX_n^{-1}\left(X_m^\mu\left|\Psi(X_m-X_m^{-\frac{1}{3}})\right| +\int_0^{X_m-X_m^{-\frac{1}{3}}}\left(2\la x\ra^\mu\left(J_1+J_3\right)+\mu\la x\ra^{\mu-1}\frac{ x}{\la x\ra}\left|\Psi(x)\right|\right)dx\right)\\
\leq& CX_n^{-1}X_m^\mu\left(\left|\Psi(X_m-X_m^{-\frac{1}{3}})\right| +\int_0^{X_m-X_m^{-\frac{1}{3}}}\left(J_1+J_3\right)dx\right)+C\mu X_n^{-1}\int_0^{X_m-X_m^{-\frac{1}{3}}}\la x\ra^{\mu-1}\left|\Psi(x)\right|dx.
\end{align*}
Clearly,
\begin{align*}
X_m^\mu\left|\Psi(X_m-X_m^{-\frac{1}{3}})\right| &\leq CX_m^\mu\left(X_m^2-(X_m-X_m^{-\frac{1}{3}})^2\right)^{-\frac{1}{4}}
 \left(X_n^2-(X_m-X_m^{-\frac{1}{3}})^2\right)^{-\frac{1}{4}}\\
&\leq C X_m^{-\frac{1}{3}+\mu},
\end{align*}
and
\begin{align*}
 & \int_0^{X_m-X_m^{-\frac13}}\mu \la x\ra^{\mu-1}\left|\Psi(x)\right|dx
\leq C\left(X_m^2-(X_m-X_m^{-\frac13})^2\right)^{-\frac14}
 \left(X_n^2-(X_m-X_m^{-\frac13})^2\right)^{-\frac14}\int_0^{X_m}\mu x^{\mu-1}dx\\
\leq& C\left(X_m^2-(X_m-X_m^{-\frac13})^2\right)^{-\frac14}
 \left(X_n^2-(X_m-X_m^{-\frac13})^2\right)^{-\frac14}X_m^{\mu}\leq CX_m^{-\frac13+\mu},
\end{align*}
together with
\begin{align*}
\int_0^{X_m-X_m^{-\frac{1}{3}}}J_1dx &\leq C \int_0^{X_m-X_m^{-\frac{1}{3}}} x(X_m^2-x^2)^{-\frac{5}{4}}(X_m^2-x^2)^{-\frac{1}{4}}dx
 \leq C X_m^{-\frac{1}{3}},
\end{align*}
and
\begin{align*}
\int_0^{X_m-X_m^{-\frac{1}{3}}}J_3dx&\leq CX_m^{-1}\int_0^{X_m-X_m^{-\frac{1}{3}}}(X_m-x)^{-3}dx \leq C X_m^{-\frac{1}{3}}.
\end{align*}
So we obtain
$\left|\int^{X_m-X_m^{-\frac{1}{3}}}_{0} \la x\ra^\mu e^{{\rm i}kx^\beta}\psi_1^{(m)}(x)\overline{\psi_1^{(n)}(x)}dx\right|\leq CX_m^{-\frac{1}{3}+\mu}X_n^{-1}.$
Now we turn to remained three terms. Since $m_0< m\le n$,
\begin{align*}
\left|\int^{X_m-X_m^{-\frac{1}{3}}}_{0} \la x\ra^\mu e^{{\rm i}kx^\beta}\psi_2^{(m)}(x)\overline{\psi_1^{(n)}(x)}dx\right|
&\leq C\int_0^{X_m-X_m^{-\frac{1}{3}}}X_m^{-2+\mu}(X_m^2-x^2)^{-\frac{1}{4}}(X_n^2-x^2)^{-\frac{1}{4}}dx\\
&\leq CX_m^{-\frac{3}{2}+\mu}X_n^{-\frac12}\le Cn^{-\frac14+\frac\mu2}.
\end{align*}
Similarly, when $m_0< m\le n$, we have
$$ \left|\int^{X_m-X_m^{-\frac13}}_{0} \la x\ra^\mu e^{{\rm i}kx^\beta}\psi_1^{(m)}(x)\overline{\psi_2^{(n)}(x)}dx\right|\leq Cn^{-1+\frac\mu2} $$
and
$$ \left|\int^{X_m-X_m^{-\frac13}}_{0} \la x\ra^\mu e^{{\rm i}kx^\beta}\psi_2^{(m)}(x)\overline{\psi_2^{(n)}(x)}dx\right|\leq Cn^{-1+\frac\mu2}.$$
Thus,
$$\left|\int^{X_m-X_m^{-\frac13}}_{0} \la x\ra^\mu e^{{\rm i}kx^\beta}h_m(x)\overline{h_n(x)}dx\right|\leq \frac{C}{n^{\frac14-\frac\mu2}}\leq \frac{C}{m^{\frac18-\frac\mu4}n^{\frac18-\frac\mu4}},\quad m_0<m\le n. $$
Case 2: $k>\frac{X_n^{2-\beta}}{8}>0$.\\
Since $m\leq n$, we have $2n\leq (8k)^{\frac{2}{2-\beta}}+1$. It follows that
$$
\left|\int_0^{X_m-X_m^{-\frac13}} \la x\ra^\mu e^{ {\rm i}kx^\beta}h_m(x)\overline{h_n(x)}dx\right| \leq CX_m^\mu\leq CX_m^\mu\frac{m^{\frac18} n^{\frac18}}{m^{\frac18}n^{\frac18}}\leq \frac{Ck^\frac{1}{4-2\beta}}{m^{\frac18-\frac\mu4} n^{\frac18-\frac\mu4}}.
$$
Combining with these two cases we finish the proof.
\end{proof}

\begin{lemma}\label{lemma2.65}
	For $k\neq 0$, if $X_n>2X_m$, then
	$$\left|\int_0^{X_m-X_m^{\frac{2}{3}}} \la x\ra^\mu e^{{\rm i}kx^{2}}h_m(x)\overline{h_n(x)}dx\right|\leq \frac{C}{m^{\frac{1}{8}-\frac\mu4} n^{\frac{1}{8}-\frac\mu4}}, $$
	where $m_0< m\leq n$.
\end{lemma}
\begin{proof}
	We first estimate the main part of the integral. By the oscillating integral estimate,
	\begin{align*}
	\int^{X_m-X_m^{\frac{2}{3}}}_{0} \la x\ra^\mu e^{{\rm i}kx^{2}}\psi_1^{(m)}(x)\overline{\psi_1^{(n)}(x)}dx= C\int^{X_m-X_m^{\frac{2}{3}}}_{0} \la x\ra^\mu e^{{\rm i}(\zeta_m-\zeta_n+kx^{2})}\Psi(x)dx,\\
	\end{align*}
	where
	$\Psi(x)=(X_m^2-x^2)^{-\frac{1}{4}} (X_n^2-x^2)^{-\frac{1}{4}}\cdot f_m(x)\overline{f_n(x)}.$
	Since
	$
	g''(x)\geq g''(0)=\frac{1}{X_m}-\frac{1}{X_n}\geq\frac12X_m^{-1},
	$
	 by Lemma \ref{oscillatory integral lemma},
	\begin{align*}
	&\left|\int^{X_m-X_m^{\frac{2}{3}}}_{0}e^{{\rm i}\frac{\zeta_m-\zeta_n+kx^2}{X_n}X_n}\la x\ra^\mu\Psi(x)dx\right|\\
	\leq& C{X_m^\frac13}\left(\left|\left(\la x\ra^\mu\Psi\right)(X_m-X_m^{\frac{2}{3}})\right|+\int_0^{X_m-X_m^{\frac{2}{3}}}\left|\left(\la x\ra^\mu\Psi\right)^\prime(x)\right|dx\right)\\
	\leq& C{X_m^{\frac13}}\left(X_m^\mu\left|\Psi(X_m-X_m^{\frac{2}{3}})\right| +\int_0^{X_m-X_m^{\frac{2}{3}}}\left(2\la x\ra^\mu\left(J_1+J_3\right)+\mu\la x\ra^{\mu-1}\frac{ x}{\la x\ra}\left|\Psi(x)\right|\right)dx\right)\\
	\leq& C{X_m^{\frac13}}X_m^\mu\left(\left|\Psi(X_m-X_m^{\frac{2}{3}})\right| +\int_0^{X_m-X_m^{\frac{2}{3}}}\left(J_1+J_3\right)dx\right)+C\mu {X_m^\frac13}\int_0^{X_m-X_m^{\frac{2}{3}}}\la x\ra^{\mu-1}\left|\Psi(x)\right|dx.
	\end{align*}
	The estimate comes from three terms. Clearly,  for $x\in[0,X_m-X_m^\frac23]$ we have 
	\[
	|\Psi(x)|\le C(X_m^2-x^2)^{-\frac14}(X_n^2-x^2)^{-\frac14}\le C(X_mX_n)^{-\frac14}(X_m-x)^{-\frac14}(X_n-x)^{-\frac14}\le CX_m^{-\frac5{12}}X_n^{-\frac12}.
	\]
It follows that 
\[ 
	X_m^\mu\left|\Psi(X_m-X_m^{\frac{2}{3}})\right| \le CX_m^{\mu-\frac5{12}}X_n^{-\frac12},
\]
	and
\[
	 \int_0^{X_m-X_m^{\frac23}}\mu \la x\ra^{\mu-1}\left|\Psi(x)\right|dx
	\leq CX_m^{-\frac5{12}}X_n^{-\frac12}\int_0^{X_m}\langle x\rangle^{\mu-1}dx\leq CX_m^{\mu-\frac5{12}}X_n^{-\frac12},
\]
	together with
\[
	\int_0^{X_m-X_m^{\frac{2}{3}}}J_1dx \leq C \int_0^{X_m-X_m^{\frac{2}{3}}} x(X_m^2-x^2)^{-\frac{5}{4}}(X_n^2-x^2)^{-\frac{1}{4}}dx
	\leq C X_m^{-\frac5{12}}X_n^{-\frac12},
\]
	and
\[
	\int_0^{X_m-X_m^{\frac{2}{3}}}J_3dx\leq CX_m^{-\frac34}X_n^{-\frac12}\int_0^{X_m-X_m^{\frac{2}{3}}}(X_m-x)^{-\frac{11}4}dx \leq C X_m^{-\frac5{12}}X_n^{-\frac12},
\]
	we obtain
	$ \left|\int^{X_m-X_m^{\frac{2}{3}}}_{0} \la x\ra^\mu e^{{\rm i}kx^2}\psi_1^{(m)}(x)\overline{\psi_1^{(n)}(x)}dx\right|\leq CX_m^{-\frac1{12}+\mu}X_n^{-\frac12}\le C(X_mX_n)^{\frac\mu2-\frac14}.$
	The estimate of rest parts of the integral is similar with Lemma \ref{lemma2.6}. Thus, 
	$$\left|\int^{X_m-X_m^{\frac23}}_{0} \la x\ra^\mu e^{ {\rm i}kx^2}h_m(x)\overline{h_n(x)}dx\right|\leq \frac{C}{m^{\frac18-\frac\mu4}n^{\frac18-\frac\mu4}},\quad m_0<m\le n. $$

\end{proof}

\begin{lemma}\label{lemma2.7}
If $X_n>2X_m$ and $1<\beta\leq 2$,  then
$$\left|\int_{X_m-X_m^{\nu_1}}^{X_m} \la x\ra^\mu e^{{\rm i}kx^\beta}h_m(x) \overline{h_n(x)}dx\right| \leq
\frac{C}{m^{\frac{3}{16}-\frac{3}{16}\nu_1-\frac\mu4} n^{\frac{3}{16}-\frac{3}{16}\nu_1-\frac\mu4}}, $$
where $m_0<m\leq n$ and 
\[
\nu_1=
\begin{cases}
	-1/3,& 1<\beta<2,\\
	2/3, &\beta=2.
\end{cases}
\]
\end{lemma}
\begin{proof}
First,
\begin{align*}
&\left|\int_{X_m-X_m^{\nu_1}}^{X_m} \la x\ra^\mu e^{{\rm i}kx^\beta}\psi_1^{(m)}(x) \overline{\psi_1^{(n)}(x)}dx\right|\\
\leq&C\int_{X_m-X_m^{\nu_1}}^{X_m} \la x\ra^\mu(X_m^2-x^2)^{-\frac{1}{4}}(X_n^2-x^2)^{-\frac{1}{4}}dx\\
\leq& C X_m^{-\frac{1}{4}+\mu} (X_n^2-X_m^2)^{-\frac{1}{4}}\int_{X_m-X_m^{\nu_1}}^{X_m}(X_m-x)^{-\frac{1}{4}}dx\\
\leq& C X_m^{-\frac{1}{4}+\mu} (X_n^2-\frac{X_n^2}{4})^{-\frac{1}{4}}X_m^{\frac{3}{4}\nu_1}\leq C X_m^{-\frac{3}{8}+\frac{3}{8}\nu_1+\frac\mu2} X_n^{-\frac{3}{8}+\frac{3}{8}\nu_1+\frac\mu2}.
\end{align*}
Similarly,
$$\left|\int_{X_m-X_m^{\nu_1}}^{X_m} \la x\ra^\mu e^{{\rm i}kx^\beta}\psi_{j_1}^{(m)}(x) \overline{\psi_{j_2}^{(n)}(x)}dx\right| \leq C X_m^{-\frac{3}{8}+\frac{3}{8}\nu_1+\frac\mu2} X_n^{-\frac{3}{8}+\frac{3}{8}\nu_1+\frac\mu2},\  \ j_1,j_2\in\{1,2\}.$$
Thus we finish the proof.
\end{proof}

\begin{lemma}\label{lemma2.8}
When $X_n > 2X_m$ and $1<\beta\leq 2$, 
$$\left|\int_{X_m}^{X_n}\la x\ra^\mu e^{{\rm i}kx^\beta}h_m(x)\overline{h_n(x)}dx\right|\leq \frac{C}{m^{\frac{1}{8}-\frac\mu4}n^{\frac{1}{12}-\frac\mu4}}, $$
where $m_0<m\leq n$. 
\end{lemma}
\begin{proof}
When $X_n>2X_{m_0}$ , $X_m+X_m^{-\frac{1}{3}}\leq  \frac{X_n}{2}+1\leq\frac{3}{4}X_n$. It follows 
\begin{align*}
&\left|\int^{X_m+X_m^{-\frac{1}{3}}}_{X_m}\la x\ra^\mu e^{{\rm i}kx^\beta}\psi_1^{(m)}(x)\overline{\psi_1^{(n)}(x)}dx\right|\\
\leq &CX_m^\mu\int^{X_m+X_m^{-\frac{1}{3}}}_{X_m}(x^2-X_m^2)^{-\frac{1}{4}} (X_n^2-x^2)^{-\frac{1}{4}}dx\\
\leq& CX_m^{-\frac{1}{4}+\mu}\left(X_n^2-(X_m+X_m^{-\frac{1}{3}})^2\right)^{-\frac{1}{4}} \int_{X_m}^{X_m+X_m^{-\frac{1}{3}}}(x-X_m)^{-\frac{1}{4}}dx\\
\leq& CX_m^{-\frac12+\frac\mu2}X_n^{-\frac12+\frac\mu2}.
\end{align*}
From $X_n> 2X_m$, we have $X_n-X_n^{-\frac13}\geq\frac32X_m$, together with Lemma 5.5 in \cite{LiangLuo2021}, thus 
\begin{align*}
&\left|\int_{X_m+X_m^{-\frac{1}{3}}}^{\frac32X_m} \la x\ra^\mu e^{{\rm i}kx^\beta}\psi_1^{(m)}(x) \overline{\psi_1^{(n)}(x)}dx\right|\\
\leq& CX_m^\mu\int_{X_m+X_m^{-\frac{1}{3}}}^{\frac32X_m} (x^2-X_m^2)^{-\frac{1}{4}} (X_n^2-x^2)^{-\frac{1}{4}}e^{{\rm i}\zeta_m}dx\\
\leq&C X_m^{-\frac{1}{4}+\mu}\left(X_n^2-(X_n-X_n^{-\frac{1}{3}})^2\right)^{-\frac{1}{4}} \int_{X_m+X_m^{-\frac{1}{3}}}^{\frac32X_m} (x-X_m)^{-\frac{1}{4}} e^{-(x-X_m)}dx\\
\leq&C X_m^{-\frac{1}{4}+\mu}X_n^{-\frac{1}{6}}\int_0^\infty t^{-\frac{1}{4}}e^{-t}dt\leq C X_m^{-\frac{1}{4}+\frac\mu2}X_n^{-\frac{1}{6}+\frac\mu2}.
\end{align*}
When $x\geq \frac32X_m$, $x-X_m\geq\frac13x$, it follows 
\begin{align*}
&\left|\int_{\frac32X_m}^{X_n-X_n^{-\frac{1}{3}}} \la x\ra^\mu e^{{\rm i}kx^\beta}\psi_1^{(m)}(x) \overline{\psi_1^{(n)}(x)}dx\right|\\
\leq& C\int_{\frac32X_m}^{X_n-X_n^{-\frac{1}{3}}} \la x\ra^\mu(x^2-X_m^2)^{-\frac{1}{4}} (X_n^2-x^2)^{-\frac{1}{4}}e^{{\rm i}\zeta_m}dx\\
\leq&C X_m^{-\frac{1}{4}}\left(X_n^2-(X_n-X_n^{-\frac{1}{3}})^2\right)^{-\frac{1}{4}} \int_{\frac32X_m}^{X_n-X_n^{-\frac{1}{3}}} (x-X_m)^{-\frac{1}{4}+\mu} e^{-(x-X_m)}dx\\
\leq&C X_m^{-\frac{1}{4}}X_n^{-\frac{1}{6}}\int_0^\infty t^{-\frac{1}{4}+\mu}e^{-t}dt\leq C X_m^{-\frac{1}{4}}X_n^{-\frac{1}{6}}.
\end{align*}
Thus, 
\begin{align*}
&\left|\int_{X_n-X_n^{-\frac{1}{3}}}^{X_n}\la x\ra^\mu e^{{\rm i}kx^\beta}\psi_1^{(m)}(x) \overline{\psi_1^{(n)}(x)}dx\right|\\
\leq& C\int_{X_n-X_n^{-\frac{1}{3}}}^{X_n}x^\mu (x^2-X_m^2)^{-\frac{1}{4}}(X_n^2-x^2)^{-\frac{1}{4}}e^{{\rm i}\zeta_m}dx\\
\leq& C  \left((X_n-X_n^{-\frac{1}{3}})^2-X_m^2\right)^{-\frac{1}{4}}X_n^{-\frac{1}{4}} \int_{X_n-X_n^{-\frac{1}{3}}}^{X_n} (X_n-x)^{-\frac{1}{4}}x^\mu e^{-\frac13x}dx\\
\leq& C X_m^{-\frac12}X_n^{-\frac{1}{4}} \int_{X_n-X_n^{-\frac{1}{3}}}^{X_n} (X_n-x)^{-\frac{1}{4}}dx\\
\leq&
 CX_m^{-\frac{1}{2}}X_n^{-\frac{1}{2}} \leq C(X_mX_n)^{-\frac12+\frac\mu2}.
\end{align*}
Combining with all the above, we have 
$\left|\int_{X_m}^{X_n}\la x\ra^\mu e^{{\rm i}kx^\beta}\psi_1^{(m)}(x) \overline{\psi_1^{(n)}(x)} dx \right|\leq \frac{C}{m^{\frac{1}{8}-\frac\mu4}n^{\frac{1}{12}-\frac\mu4}}$.  The rest estimates are similar as above.  
\end{proof}
Combining with Lemma \ref{lemma2.6}, \ref{lemma2.65}, \ref{lemma2.7}, \ref{lemma2.8}, we finish the proof of Lemma \ref{lemma2.4}.

\subsubsection{The integral estimate on $[0, X_n]$ when  $X_m\leq X_n\leq 2X_m$.}\label{sub4.2.3}
 One can split the integral into
$$\int^{X_n}_0 \la x\ra^\mu e^{{\rm i}kx^{\beta}}h_m(x)\overline{h_n(x)}dx = \left(\int^{X_m^\frac23}_0+\int_{X_m^\frac23}^{X_m-X_m^{\nu_2}}+ \int_{X_m-X_m^{\nu_2}}^{X_n}\right) \la x\ra^\mu e^{{\rm i}kx^\beta}h_m(x)\overline{h_n(x)}dx,$$
and estimate them respectively, where 
\[
\nu_2=
\begin{cases}
	1-\frac{\beta}{3},& 1<\beta<2,\\
	\frac59, &\beta=2.
\end{cases}
\]
Our main aim in this part  is to build the following two lemmas. 
\begin{lemma}\label{4.2.3}
For $X_m \leq X_n \leq 2X_m$ and $k\neq 0$, 
$$\displaystyle\left|\int_0^{X_n} \la x\ra^\mu e^{{\rm i}kx^2}h_m(x)\overline{h_n(x)}dx\right|\leq \frac{ C(|k|^{-1}\vee 1)}{m^{\frac{1}{18}-\frac\mu4} n^{\frac{1}{18}-\frac\mu4}},$$
where $C>0$, $m_0< m\leq n$.
\end{lemma}
\begin{lemma}\label{4.2.3-1}
For $X_m \leq X_n \leq 2X_m$, $k\neq 0$ and $1<\beta<2$, 
$$\displaystyle\left|\int_0^{X_n} \la x\ra^\mu e^{{\rm i}kx^{\beta}}h_m(x)\overline{h_n(x)}dx\right|\leq \frac{C(|k|^{-1}\vee |\beta(\beta-1)(\beta-2)k|^{-\frac13}\vee 1)}{m^{\frac{\beta}{24}-\frac\mu4} n^{\frac{\beta}{24}-\frac\mu4}},$$
where $C>0$, $m_0< m\leq n$.
\end{lemma}

From a straightforward computation we have 
\begin{lemma}\label{lemma2.9}
For $X_m \leq X_n \leq 2X_m$ and $1<\beta\leq 2$, 
$$\displaystyle\left|\int_0^{X_m^\frac23} \la x\ra^\mu e^{{\rm i}kx^\beta}h_m(x)\overline{h_n(x)}dx\right|\leq \frac{C}{m^{\frac{1}{12}-\frac\mu4} n^{\frac{1}{12}-\frac\mu4}},$$
where $C>0$, $m_0< m\leq n$.
\end{lemma}
\indent Next we estimate the integral on $[X_m^\frac23,X_m-X_m^{\frac59}]$, for which we discuss different cases as the following.
\begin{lemma}\label{lemma2.10}
If $X_m\leq X_n\leq 2X_m$, when $k>0$ and $0\leq X_n^2-X_m^2\leq kX_m^\frac43$, then
$$ \displaystyle\left|\int_{X_m^\frac23}^{X_m-X_m^\frac59} \la x\ra^\mu e^{{\rm i}kx^2}h_m(x)\overline{h_n(x)}dx\right|\leq \frac{Ck^{-\frac12 }}{m^{\frac{7}{36}-\frac\mu4}n^{\frac{7}{36}-\frac\mu4}}, $$
where $m_0< m\leq n$.
\end{lemma}

\begin{proof}
We first estimate
$$\left|\displaystyle\int_{X_m^\frac23}^{X_m-X_m^\frac59} \la x\ra^\mu e^{{\rm i}kx^2}\psi^{(m)}_1(x)\overline{\psi^{(n)}_1(x)}dx\right| = \left|C\displaystyle\int_{X_m^\frac23}^{X_m-X_m^\frac59} \la x\ra^\mu e^{{\rm i}(\zeta_m-\zeta_n+kx^2)}\Psi(x)dx\right|.$$
Notice that
\begin{align*}
g'(x)&\leq\frac{kX_m^\frac43X_m}{\sqrt{2X_m^{\frac{14}{9}}-X_m^{\frac{10}{9}}}\sqrt{2X_m^{\frac{14}{9}}-X_m^{\frac{10}{9}}}(\sqrt{2X_m^{\frac{14}{9}}-X_m^{\frac{10}{9}}}+\sqrt{2X_m^{\frac{14}{9}}-X_m^{\frac{10}{9}}})}-2k\\
                      &\leq\frac{k X_m^\frac73}{2X_m^\frac{7}{3}}-2k= -\frac{3}{2}k,
\end{align*}
and by straightforward computation, $g''(x)\geq0$. It follows  $|g'(x)|\geq\frac{3}{2}k$ on $x\in[X_m^\frac23,X_m-X_m^\frac59]$. By Lemma \ref{oscillatory integral lemma},
\begin{align*}
&\left| \int_{X_m^\frac23}^{X_m-X_m^\frac59} \la x\ra^\mu e^{{\rm i}\frac{2}{3k}(\zeta_m-\zeta_n+kx^2)\cdot\frac{3k}{2}} (X_m^2-x^2)^{-\frac{1}{4}} (X_n^2-x^2)^{-\frac{1}{4}}\cdot f_m(x)\overline{f_n(x)}dx\right|\\
\leq& Ck^{-\frac12}\left[\left|(\la x\ra^\mu \Psi)(X_m-X_m^{\frac59})\right| + \int_{X_m^\frac23}^{X_m-X_m^\frac59} \left|(\la x\ra^\mu \Psi)^\prime(x)\right|dx\right]\\
\le&  Ck^{-\frac12}\left[ X_m^{\mu}\left|\Psi(X_m-X_m^{\frac59})\right| + \la X_m\ra^\mu\int_{X_m^\frac23}^{X_m-X_m^\frac59} \left|\Psi^\prime(x)\right|dx +\mu\int_{X_m^\frac23}^{X_m-X_m^\frac59} \la x\ra^{\mu-1}\left|\Psi(x)\right|dx \right].
\end{align*}
We estimate the above terms one by one.  Clearly,  
\begin{align*}
\left|\Psi(X_m-X_m^{\frac59})\right| &\leq C\left(X_m^2-(X_m-X_m^{\frac59})^2\right)^{-\frac{1}{4}} \left(X_n^2-(X_m-X_m^{\frac59})^2\right)^{-\frac{1}{4}} \leq C X_m^{-\frac79},
\end{align*}
and
$\int_{X_m^\frac23}^{X_m-X_m^{\frac59}} \la x\ra^{\mu-1}\left|\Psi(x)\right|dx\le CX_m^{-\frac79+\mu}$.
Besides, we have $|\Psi^\prime(x)|\le C(J_1+J_3)$ and
\begin{align*}
\int_{X_m^\frac23}^{X_m-X_m^\frac59}J_1dx &\leq  C \int_{X_m^\frac23}^{X_m-X_m^\frac59} x(X_m^2-x^2)^{-\frac{5}{4}}(X_n^2-x^2)^{-\frac{1}{4}}dx \leq C X_m^{-\frac79},
\end{align*}
also,
$\int_{X_m^\frac23}^{X_m-X_m^{\frac59}}J_3dx\leq C X_m^{-\frac79}$.
Combining with all the estimates, we obtain
$$\left|\displaystyle\int_{X_m^\frac23}^{X_m-X_m^\frac59} \la x\ra^\mu e^{{\rm i}kx^{2}}\psi^{(m)}_1(x)\overline{\psi^{(n)}_1(x)}dx\right| \leq Ck^{-\frac12} X_m^{-\frac79+\mu}.$$
The estimates for the rest three terms are easier. In fact, by $m > m_0$,
\begin{align*}
&\left|\int_{X_m^\frac23}^{X_m-X_m^\frac59} \la x\ra^\mu e^{{\rm i}kx^2}\psi^{(m)}_2(x)\overline{\psi^{(n)}_1(x)}dx\right| \leq CX_m^{-2+\mu}\int_{X_m^\frac23}^{X_m-X_m^\frac59}(X_m^2-x^2)^{-\frac{1}{2}}dx\\
&\leq CX_m^{-\frac52+\mu} \int_{0}^{X_m-X_m^\frac59}(X_m-x)^{-\frac{1}{2}}dx \leq CX_m^{-\frac79+\mu}.
\end{align*}
The estimates for the other two are similar. Thus, 
\begin{align*}
\left|\displaystyle\int_{X_m^\frac23}^{X_m-X_m^\frac59}\la x\ra^\mu e^{{\rm i}kx^2}h_m(x)\overline{h_n(x)}dx\right|
\leq& \frac{Ck^{-\frac12 }}{m^{\frac{7}{36}-\frac\mu4}n^{\frac{7}{36}-\frac\mu4}}.
\end{align*}
\end{proof}

\begin{lemma}\label{lemma2.11}
For $k>0$, $X_m\leq X_n\leq 2X_m$, if $kX_m^\frac43\leq X_n^2-X_m^2$, then
$$\displaystyle\left|\int_{X_m^\frac23}^{X_m-X_m^\frac59} \la x\ra^\mu e^{{\rm i}kx^2}h_m(x)\overline{h_n(x)}dx\right|\leq \frac{Ck^{-\frac13}}{m^{\frac{1}{18}-\frac\mu4} n^{\frac{1}{18}-\frac\mu4}},$$
where $m_0< m\leq n$.
\end{lemma}

\begin{proof}
For $kX_m^\frac43\leq X_n^2-X_m^2$, straightforward computation shows that $g'''(x)\geq0$. So
\begin{align*}
g''(x)\geq g''(0)=\frac{1}{X_m}-\frac{1}{X_n}=\frac{X_n^2-X_m^2}{X_mX_n(X_m+X_n)}\geq\frac{kX_m^\frac43}{3X_mX_mX_n}\geq\frac{k}{6}X_m^{-\frac53}.
\end{align*}

Then by Lemma \ref{oscillatory integral lemma},
\begin{eqnarray*}
&&\Big|\int_{X_m^\frac{2}{3}}^{X_m-X_m^{\frac59}}\la x\ra^\mu e^{{\rm i}\frac{\zeta_m-\zeta_n+kx^2}{kX_m^{-\frac53}} kX_m^{-\frac53}}(X_n^2-x^2)^{-\frac{1}{4}}(X_m^2-x^2)^{-\frac{1}{4}}f_m(x)\overline{f_n(x)}dx\Big| \nonumber \\
&\leq&Ck^{-\frac{1}{3}}X_m^{\frac59}\left[\left|(\la x\ra^\mu\Psi)(X_m-X_m^{\frac59})\right| + \int_{X_m^\frac{2}{3}}^{X_m-X_m^{\frac59}} \left|(\la x\ra^\mu\Psi)^\prime(x)\right|dx\right] \nonumber \\
&\leq&Ck^{-\frac{1}{3}}X_m^{\frac59}\bigg[X_m^\mu\left|\Psi(X_m-X_m^{\frac59})\right| +  X_m^\mu\int_{X_m^\frac{2}{3}}^{X_m-X_m^{\frac59}}\left|\Psi^\prime(x)\right|dx \nonumber \\
&&
+\mu\int_{X_m^\frac{2}{3}}^{X_m-X_m^{\frac59}} \la x\ra^{\mu-1}\left|\Psi(x)\right|dx\bigg].
\end{eqnarray*}
The rest part of the proof is similar with Lemma \ref{lemma2.10}.
\end{proof}

\begin{lemma}\label{lemma2.16}
For $\forall k <  0, X_m\leq X_n\leq 2X_m$, we have
$$\displaystyle\left|\int_{X_m^\frac23}^{X_m-X_m^{\nu_2}} \la x\ra^\mu e^{{\rm i}kx^\beta}h_m(x)\overline{h_n(x)}dx\right|\leq \frac{C(|k|^{-1}\vee 1)}{m^{\frac{1}{6}\beta-\frac{1}{24}+\frac{\nu_2}{8}-\frac\mu4 } n^{\frac{1}{6}\beta-\frac{1}{24}+\frac{\nu_2}{8}-\frac\mu4}},$$
where $m_0<m\leq n$. 
\end{lemma}
\begin{proof}
We first estimate
$$\left|\int_{X_m^\frac{2}{3}}^{X_m-X_m^{\nu_2}}
\la x\ra^\mu e^{{\rm i} kx^\beta}\psi_1^{(m)}(x)\overline{\psi_1^{(n)}(x)}dx\right| =\left|C\int_{X_m^\frac{2}{3}}^{X_m-X_m^{\nu_2}}\la x\ra^\mu e^{{\rm i}(\zeta_m-\zeta_n+kx^\beta)}\Psi(x)dx\right|.$$
Notice that
$g(x)=\sqrt{X_n^2-x^2}-\sqrt{X_m^2-x^2}-\beta kx^{\beta-1} \geq - kX_m^{\frac{2}{3}(\beta-1)} $
and $g^\prime(x)>0$,  then by Lemma \ref{oscillatory integral lemma},
\begin{align*}
&\left|\int_{X_m^\frac{2}{3}}^{X_m-X_m^{\nu_2}}\la x\ra^\mu e^{{\rm i}\frac{\zeta_m-\zeta_n+kx^\beta}{|k|X_m^{\frac{2}{3}(\beta-1)}}|k|X_m^{\frac{2}{3}(\beta-1)}} (X_m^2-x^2)^{-\frac{1}{4}}(X_n^2-x^2)^{-\frac{1}{4}}f_m(x)\overline{f_n(x)}dx\right|\\
\leq& \frac{C}{|k|}X_m^{-\frac{2}{3}(\beta-1)}\left[\left|(\la x\ra^\mu \Psi)(X_m-X_m^{\nu_2})\right| +\int_{X_m^\frac{2}{3}}^{X_m-X_m^{\nu_2}}|(\la x\ra^\mu \Psi)^\prime(x)|dx\right]\\
\le& C|k|^{-1}X_m^{-\frac{2}{3}(\beta-1)}\bigg[X_m^\mu \left|\Psi(X_m-X_m^{\nu_2})\right| +X_m^\mu\int_{X_m^\frac23}^{X_m-X_m^{\nu_2}}|\Psi^\prime(x)|dx \\
&+\mu\int_{X_m^\frac23}^{X_m-X_m^{\nu_2}}\la x\ra^{\mu-1}|\Psi(x)|dx\bigg].
\end{align*}
The rest part of the proof is similar with Lemma \ref{lemma2.10}.
\end{proof}

Finally, we have
\begin{lemma}\label{lemma2.14} 
For $k>0, 1<\beta<2, X_m\leq X_n\leq 2X_m$, we have
$$ \displaystyle\left|\int_{X_m^\frac23}^{X_m-X_m^{\nu_2}} \la x\ra^\mu e^{{\rm i}kx^\beta}h_m(x)\overline{h_n(x)}dx\right|\leq \frac{C|\beta(\beta-1)(\beta-2)k|^{-\frac{1}{3}} }{m^{\frac{\beta}{24}-\frac\mu4} n^{\frac{\beta}{24}-\frac\mu4}}, $$
where $m_0<m\leq n$.
\end{lemma}
\begin{proof}
First we estimate
$$\left|\displaystyle\int_{X_m^\frac23}^{X_m-X_m^{\nu_2}} \la x\ra^\mu e^{{\rm i}kx^\beta}\psi^{(m)}_1(x)\overline{\psi^{(n)}_1(x)}dx\right| = \left|C\displaystyle\int_{X_m^\frac23}^{X_m-X_m^{\nu_2}} \la x\ra^\mu e^{{\rm i}(\zeta_m-\zeta_n+kx^\beta)}\Psi(x)dx\right|.$$
Since
\begin{align*}
g''(x)\geq-\beta(\beta-1)(\beta-2)kx^{\beta-3}\geq-\beta(\beta-1)(\beta-2)kX_m^{\beta-3},
\end{align*}
then by Lemma \ref{oscillatory integral lemma},
\begin{align*}
&\left| \int_{X_m^\frac23}^{X_m-X_m^{\nu_2}} \la x\ra^\mu e^{{\rm i}(\zeta_m-\zeta_n+kx^\beta)} (X_m^2-x^2)^{-\frac{1}{4}} (X_n^2-x^2)^{-\frac{1}{4}}\cdot f_m(x)\overline{f_n(x)}dx\right|\\
\leq& C|\beta(\beta-1)(\beta-2)k|^{-\frac{1}{3}}X_m^{1-\frac{\beta}{3}}\left[\left|(\la x\ra^\mu \Psi)(X_m-X_m^{\nu_2})\right| + \int_{X_m^\frac23}^{X_m-X_m^\nu} \left|(\la x\ra^\mu \Psi)^\prime(x)\right|dx\right]\\
\leq&C|\beta(\beta-1)(\beta-2)k|^{-\frac{1}{3}}X_m^{1-\frac{\beta}{3}}\bigg[X_m^\mu\left|\Psi(X_m-X_m^{\nu_2})\right| + X_m^\mu\int_{X_m^\frac23}^{X_m-X_m^{\nu}}\left|\Psi^\prime(x)\right|dx\\
&+ \mu\int_{X_m^\frac23}^{X_m-X_m^{\nu_2}} \la x\ra^{\mu-1}\left|\Psi(x)\right|dx\bigg].
\end{align*}
The rest part of the proof is similar with Lemma \ref{lemma2.10}.
\end{proof}

For the last part of the integral, we have
\begin{lemma}\label{lemma2.18}
For $\forall k \ne 0, X_m\leq X_n\leq 2X_m$, $1<\beta\leq 2$, 
$$\displaystyle\left|\int_{X_m-X_m^{\nu_2}}^{X_m}\la x\ra^\mu e^{{\rm i}kx^{\beta}}h_m(x)\overline{h_n(x)}dx\right|\leq \frac{C}{m^{\frac{1}{8}-\frac{\nu_2}{8}-\frac\mu4} n^{\frac{1}{8}-\frac{\nu_2}{8}-\frac\mu4}}.$$
Here $m_0< m\leq n$.
\end{lemma}
\begin{proof}
First,
\begin{align*}
\left|\int_{X_m-X_m^{\nu_2}}^{X_m} \la x\ra^\mu e^{{\rm i}kx^\beta}\psi_1^{(m)}(x)\overline{\psi_1^{(n)}(x)}dx\right| &\leq C\int_{X_m-X_m^{\nu_2}}^{X_m}\la x\ra^\mu(X_m^2-x^2)^{-\frac{1}{4}}(X_n^2-x^2)^{-\frac{1}{4}}dx\\
&\leq CX_m^\mu\int_{X_m-X_m^{\nu_2}}^{X_m}(X_m^2-x^2)^{-\frac{1}{2}}dx\\
&\leq CX_m^{-\frac{1}{2}+\mu}\int_{X_m-X_m^{\nu_2}}^{X_m}(X_m-x)^{-\frac{1}{2}}dx\\
&\leq CX_m^{-\frac{1}{2}+\mu}X_m^{\frac{\nu_2}{2}}\leq \frac{C}{m^{\frac{1}{8}-\frac\mu4-\frac{\nu_2}{8}}n^{\frac{1}{8}-\frac\mu4-\frac{\nu_2}{8}}}.
\end{align*}
It follows 
$\displaystyle\left|\int_{X_m-X_m^{\nu_2}}^{X_m} \la x\ra^\mu e^{{\rm i}kx^\beta}h_m(x)\overline{h_n(x)}dx\right|\leq \frac{C}{m^{\frac{1}{8}-\frac\mu4-\frac{\nu_2}{8}}n^{\frac{1}{8}-\frac\mu4-\frac{\nu_2}{8}}}$.
\end{proof}
\begin{lemma}
For $k\neq 0,\ X_m\leq X_n\leq 2X_m$ and $1<\beta\leq 2$, we have
$$\Big|\int_{X_m}^{X_n} \la x\ra^\mu e^{{\rm i}kx^\beta}h_m(x)\overline{h_n(x)}dx\Big|\leq \frac{C}{m^{-\frac{\nu_2}{16}+\frac{1}{8}-\frac\mu4}
n^{-\frac{\nu_2}{16}+\frac{1}{8}-\frac\mu4}}.$$
\end{lemma}
\begin{proof}
For the integral on $[X_m,X_n]$, we discuss two different cases:\\
Case 1.$X_n - X_n^{-\nu_2} \ge X_m + X_m^{-\nu_2}$. We split the integral into  three parts. First,
\begin{align*}
&\left|\int^{X_m+X_m^{-\nu_2}}_{X_m} \la x\ra^\mu e^{{\rm i}kx^\beta}\psi_1^{(m)}(x) \overline{\psi_1^{(n)}(x)}dx\right|\\
\leq&\int^{X_m+X_m^{-\nu_2}}_{X_m} \la x\ra^\mu (x^2-X_m^2)^{-\frac{1}{4}}(X_n^2-x^2)^{-\frac{1}{4}}dx\\
\leq& C X_m^\mu X_m^{-\frac{1}{4}} X_n^{-\frac{1}{4}} (X_n-X_m-X_m^{-\nu_2})^{-\frac{1}{4}}\int^{X_m+X_m^{-\nu_2}}_{X_m}(x-X_m)^{-\frac{1}{4}}dx\\
\leq& C X_m^\mu X_m^{-\frac{1}{4}} X_n^{-\frac{1}{4}} X_n^{\frac{\nu_2}{4}}X_m^{-\frac{3}{4}\nu_2} \leq C n^{-\frac{1}{8}-\frac{\nu_2}{8}+\frac\mu4}m^{-\frac{1}{8}-\frac{\nu_2}{8}+\frac\mu4}.
\end{align*}
When $x \geq X_m+X_m^{-\nu_2}$,  we have ${\rm i}\zeta_m \leq -(x-X_m)$. It follows 
\begin{align*}
&\left|\int_{X_m+X_m^{-\nu_2}}^{X_n-X_n^{-\nu_2}} \la x\ra^\mu e^{{\rm i}kx^\beta}\psi_1^{(m)}(x) \overline{\psi_1^{(n)}(x)}dx\right|\\
\leq& C\int_{X_m+X_m^{-\nu_2}}^{X_n-X_n^{-\nu_2}} \la x\ra^\mu (x^2-X_m^2)^{-\frac{1}{4}} (X_n^2-x^2)^{-\frac{1}{4}}e^{ {\rm i}\zeta_m}dx\\
\leq&C (2X_m)^\mu X_m^{-\frac{1}{4}}\left(X_n^2-(X_n-X_n^{-\nu_2})^2\right)^{-\frac{1}{4}} \int_{X_m+X_m^{-\nu_2}}^{X_n-X_n^{-\nu_2}} (x-X_m)^{-\frac{1}{4}} e^{ {\rm i}\zeta_m}dx\\
\leq&C (2X_m)^\mu X_m^{-\frac{1}{4}}X_n^{\frac{\nu_2-1}{4}}\int_0^\infty t^{-\frac{1}{4}}e^{-t}dt\leq C n^{-\frac{1}{8}+\frac{\nu_2}{16}+\frac\mu4}m^{-\frac{1}{8}+\frac{\nu_2}{16}+\frac\mu4}.
\end{align*}
Finally, from 
\begin{align*}
&\left|\int_{X_n-X_n^{-\nu_2}}^{X_n}\la x\ra^\mu e^{{\rm i}kx^\beta}\psi_1^{(m)}(x) \overline{\psi_1^{(n)}(x)}dx\right|\\
\leq&\int_{X_n-X_n^{-\nu_2}}^{X_n} \la x\ra^\mu (x^2-X_m^2)^{-\frac{1}{4}}(X_n^2-x^2)^{-\frac{1}{4}}dx\\
\leq&C (2X_m)^\mu\left((X_n-X_n^{-\nu_2})^2-X_m^2\right)^{-\frac{1}{4}}X_n^{-\frac{1}{4}} \int_{X_n-X_n^{-\nu_2}}^{X_n} (X_n-x)^{-\frac{1}{4}}dx\\
\leq& C n^{-\frac{1}{8}-\frac{\nu_2}{8}+\frac\mu4}m^{-\frac{1}{8}-\frac{\nu_2}{8}+\frac\mu4},
\end{align*}
it follows 
$\Big|\int_{X_m}^{X_n} \la x\ra^\mu e^{{\rm i}kx^\beta}h_m(x)\overline{h_n(x)}dx\Big|\leq \frac{C}
{(nm)^{\frac{1}{8}-\frac\mu4-\frac{\nu_2}{16}}}.$\\
Case 2. $X_n - X_n^{-\nu_2} < X_m + X_m^{-\nu_2}$. In fact, notice that the function  $(x-X_m)^{-\frac{1}{4}}(X_n-x)^{-\frac{1}{4}}$ is symmetric on $[X_m,X_n]$, we obtain 
\begin{align*}
&\left|\int_{X_m}^{X_n} \la x\ra^\mu e^{{\rm i}kx^\beta}\psi_1^{(m)}(x) \overline{\psi_1^{(n)}(x)}dx\right|\\
\leq&\int_{X_m}^{X_n} \la x\ra^\mu(x^2-X_m^2)^{-\frac{1}{4}}(X_n^2-x^2)^{-\frac{1}{4}}dx\\
\leq&C (2X_m^\mu)X_m^{-\frac{1}{4}}X_n^{-\frac{1}{4}}\int_{X_m}^{X_n} (x-X_m)^{-\frac{1}{4}}(X_n-x)^{-\frac{1}{4}}dx\\
\leq&C (2X_m^\mu)X_m^{-\frac{1}{4}}X_n^{-\frac{1}{4}}\int_{X_m}^{X_n} (x-X_m)^{-\frac{1}{2}}dx\\
\leq&C X_m^\mu X_m^{-\frac{1}{4}}X_n^{-\frac{1}{4}}(X_n-X_m)^\frac{1}{2}\leq (nm)^{\frac14 \mu-\frac18-\frac14\nu_2}.
\end{align*}
Thus,
$\displaystyle\left|\int_{X_m}^{X_n} \la x\ra^\mu e^{{\rm i}kx^\beta}h_m(x)\overline{h_n(x)}dx\right|\leq\frac{C}{m^{\frac18+\frac{\nu_2}{8}-\frac\mu4}n^{\frac18+\frac{\nu_2}{8}-\frac\mu4}}.$ 
Combining with the above two cases,  we finish the proof. 
\end{proof}

Lemma \ref{4.2.3} and \ref{4.2.3-1} follow directly by  the lemmas in subsection \ref{sub4.2.3}.  Combining with all the lemmas in this section we finish the proof of Lemma \ref{CriticalLemma}
for $1<\beta\leq 2$. 

\section{Proof of Lemma \ref{CriticalLemma} when $\beta>2$}\label{S5}
In the following we will suppose  that  $m\le n$ without losing the generality. As the case $1<\beta\leq 2$ we only need to estimate the integral on $[0,\infty]$. 
We first apply Theorem 3.1 in \cite{KT2005} to obtain the integral estimates on $[2X_m,\infty)$ as follows.
\begin{lemma}
For $\mu\ge0$, then we have
\[
\left|\int_{2X_m}^\infty\langle x\rangle^\mu e^{\rmi kx^\beta}h_m(x)\overline{{h}_n(x)}dx\right|
\le\frac C{(mn)^{\frac14\left(\frac13-\mu\right)}}.
\]
\end{lemma}
\begin{proof}
By Theorem 3.1 in \cite{KT2005} we have 
\begin{align*}
&\left|\int_{2X_m}^\infty\langle x\rangle^\mu e^{\rmi kx^\beta}h_m(x)\overline{{h}_n(x)}dx\right|\\
\le&C\int_{2X_m}^\infty \langle x\rangle^{-1} \langle x\rangle^{\mu+1}|h_m(x)||h_n(x)|dx\\
\le&C\|\langle x\rangle^{-1}\|_{L^2(\mathbb R)}^\frac12\cdot\|\langle x\rangle^{\mu+1} h_m(x)\|_{L^2(\{x:|x|\ge2X_m\})}^\frac12\cdot\|h_n\|_{L^\infty(\mathbb R)}\\
\le&CX_m^{-\frac16}X_n^{-\frac16}\le C(mn)^{\frac14\left(\mu-\frac13\right)}. 
\end{align*}
\end{proof}

Next we consider the integral on $[0,2X_m]$.  Define $\nu_3=\frac{5\beta-4}{2(\beta-1)(2\beta-1)}\in(0,1)$ when $\beta>2$. In the following  we will discuss two different cases depending on whether $X_n^{\nu_3}\ge2X_m$ or not.

\textbf{The case $X_n^{\nu_3}\ge 2X_m$}. In this case we directly estimate the rest  integral on $[0,2X_m]$.
\begin{lemma}\label{between12first}
For $X_n^{\nu_3}\geq 2X_m$ and $\beta>2$, then
$$\displaystyle\left|\int_{0}^{2X_m} \la x\ra^\mu e^{{\rm i}kx^{\beta}}h_m(x)\overline{h_n(x)}dx\right|\leq \frac C {(mn)^{\frac{1}{4}\left(\frac{\beta-2}{4\beta-2}-\mu\right)}}. $$
\end{lemma}

\begin{proof}
	From Lemma \ref{LP eigenfucntion bounds},
\begin{align*}
&\left|\int_{0}^{2X_m} \la x\ra^\mu e^{{\rm i}kx^{\beta}}h_m(x) \overline{h_n(x)}dx\right|\\
\leq&\int_{0}^{2X_m} \la x\ra^\mu|h_m(x)||h_n(x)|dx \leq C X_m^\mu \int_{0}^{X_n^{\nu_3}} |h_m(x)||h_n(x)|dx\\
\leq&C X_m^\mu\|h_m(x)\|_{L^2}\|h_n(x)\|_{L^p}\left(\int_0^{X_n^{\nu_3}}dx\right)^{\frac{1}{q}}\\
\leq & C X_m^{\mu}X_n^{-\left(\frac{1}{2}-\frac{1}{p}\right)}X_n^{\frac{\nu_3}{q}} \leq C (mn)^{\frac{1}{4}\left(\mu-\frac{\beta-2}{4\beta-2}\right)},
\end{align*}
where $\frac{1}{p}+\frac{1}{q}=\frac{1}{2},q=\frac{4\beta-3}{\beta-1}>4.$
\end{proof}

\textbf{The case $X_n^{\nu_3}<2X_m$}. In this case we split the rest integral into two parts as follows.
\begin{lemma}
	For $k\ne0, X_n^{\nu_3}< 2X_m$ and  $\beta>2$, then
	$$\displaystyle\left|\int_{X_n^{\nu_3}}^{2X_m} \la x\ra^\mu e^{{\rm i}kx^{\beta}}h_m(x)\overline{h_n(x)}dx\right|\leq \frac{C(|k\beta|^{-1}\vee 1)}{(mn)^{\frac{1}{4}\left(\frac{\beta-2}{4\beta-2}-\mu\right)}}. $$
\end{lemma}

\begin{proof}
	Denote $\psi(x)=\la x\ra^{\mu}h_m(x)h_n(x),\phi(x)=x^\beta$. Notice that when $x\in[X_n^{\nu_3},2X_m]$, $|\phi'(x)|\geq\beta X_n^{\nu_3(\beta-1)}$. So by Lemma \ref{oscillatory integral lemma},
	\begin{align*}
	&\left|\int_{X_n^{\nu_3}}^{2X_m} \la x\ra^\mu e^{{\rm i}kx^{\beta}}h_m(x) \overline{h_n(x)}dx\right|\\
		\leq& C|k\beta|^{-1}X_n^{-\nu_3(\beta-1)}\left[\left| \psi(2X_m)\right| + \int_{X_n^{\nu_3}}^{2X_m} \left|\psi^\prime(x)\right|dx\right],
	\end{align*}
	where
	$
	\left|\psi(2X_m)\right| \leq C X_m^{\mu}$, 
	and
	$|\int_{X_n^{\nu_3}}^{2X_m} \la x\ra^{\mu-1}h_m(x)\overline{h_n(x)}dx|\le CX_m^{\mu-1}$.
	By Lemma \ref{LP eigenfucntion bounds}, $\|h_m'(x)\|_{L^2}\le CX_m$. Thus,
	\begin{gather*}
	\left|\int_{X_n^{\nu_3}}^{2X_m} \la x\ra^{\mu}h_m'(x)\overline{h_n(x)}dx\right|\le CX_m^{\mu}X_m=CX_m^{\mu+1},\\
	\left|\int_{X_n^{\nu_3}}^{2X_m} \la x\ra^{\mu}h_m(x)\overline{h_n'(x)}dx\right|\le CX_m^{\mu}X_n\le CX_m^{\mu}X_n.
	\end{gather*}
	Combining with all the conclusions we have
	\begin{align*}
	&\left|\int_{X_n^{\nu_3}}^{2X_m} \la x\ra^\mu e^{{\rm i}kx^{\beta}}h_m(x) \overline{h_n(x)}dx\right|\\
	\leq&C|k\beta|^{-1}X_n^{-\nu_3(\beta-1)+1} X_m^{\mu}\le   C |k\beta|^{-1} (X_mX_n)^\frac\mu2X_n^{-\frac{\beta-2}{4\beta-2}}\\
	 \leq &C |k\beta|^{-1}(mn)^{\frac{1}{4}\left(\mu-\frac{\beta-2}{4\beta-2}\right)}.
	\end{align*}
\end{proof}

\begin{lemma}
	For $X_n^{\nu_3}< 2X_m$ and  $\beta>2$, then
\[
\left|\int_{0}^{X_n^{\nu_3}} \la x\ra^\mu e^{{\rm i}kx^{\beta}}h_m(x)\overline{h_n(x)}dx\right|\leq \frac{C}{(mn)^{\frac{1}{4}\left(\frac{\beta-2}{4\beta-2}-\mu\right)}}.
\]
\end{lemma}
\begin{proof}
The proof is similar as Lemma \ref{between12first}, we omit it.
\end{proof}
Hence, combining the above four Lemmas we obtain Lemma \ref{CriticalLemma} when $\beta>2$. 

\noindent
\textbf{Acknowledgements}\addcontentsline{toc}{section}{Acknowledgements}
The authors were partially supported by Natural Science Foundation of Shanghai (19ZR1402400) and NSFC grant (12071083).

\section{Appendix}\label{S6}

\indent  In the following we will introduce two technical lemmas  without proof. The first lemma  provides an estimate of oscillatory  integral.  For more details, see \cite{Stein}.
\begin{lemma}(\cite{Stein})\label{oscillatory integral lemma}
Suppose $\phi$ is real-valued and smooth in $(A,B)$, $\psi$ is complex-valued, and that $|\phi^{(k)}(x)|\geq1$ for all $x\in(A,B)$. Then
$$\left|\int_A^B e^{\rm i\lambda\phi(x)}\psi(x)dx\right|\leq c_k\lambda^{-1/k}\left[|\psi(B)|+\int_A^B|\psi^\prime(x)|dx\right]$$
holds when:\\
(i) $k\geq 2$, or\\
(ii) $k=1$ and $\phi^\prime(x)$ is monotonic.\\
The bound $c_k$ is independent of $\phi$, $\psi$ and $\lambda$.
\end{lemma}
The second lemma shows that the $L^p$-norm of the eigenfunction of harmonic oscillating operator can be controlled by its $L^2$-norm.
\begin{lemma}[\cite{KT2005}]\label{LP eigenfucntion bounds}
	Sppose that $h(x)$ is the eigenfunction of Herimite operator with the corresponding eigenvalue $\mu^2$. Then
	$$\|h\|_{L^p}\leq\mu^{\rho(p)}\|h\|_{L^2},$$
	where
\[
	\rho(p)=
	\begin{cases}
	-(\frac{1}{2}-\frac1p), & 2\leq p<4,\\
	-\frac13+\frac13(\frac12-\frac1p),&4<p\leq\infty.
	\end{cases}
\]
\end{lemma}

\end{document}